\documentclass[leqno]{siamart1116}
\pdfoutput=1

\usepackage[algo2e,linesnumbered,vlined,ruled]{algorithm2e}
\usepackage{subfigure}
\usepackage{tikz}

\usepackage{comment}
\usepackage{times}
\usepackage{amsmath,amsxtra,amsfonts,amscd,amssymb,bm}
\usepackage{color}
\usepackage{exscale}
\usepackage{pgfplots}
\usepackage{multirow} 

\newtheorem{assumption}[theorem]{Assumption}

\newcommand{\be}{\begin{equation}}
\newcommand{\ee}{\end{equation}}
\newcommand{\bee}{\begin{equation*}}
\newcommand{\eee}{\end{equation*}}
\newcommand{\bea}{\begin{eqnarray}}
\newcommand{\eea}{\end{eqnarray}}
\newcommand{\beaa}{\begin{eqnarray*}}
	\newcommand{\eeaa}{\end{eqnarray*}}

\newcommand{\st}{\quad \textrm{s.t.}\quad}  
\newcommand{\R}{\mathbb{R}}

\newcommand{\cU}{\mathcal{U}}

\newcommand{\iprod}[2]{\left\langle{#1},{#2}\right\rangle}

\newcommand{\cpn}{\mathcal{CP}_n}  
\newcommand{\half}{\frac{1}{2}}

\newcommand{\mak}{\mathcal{A}^{n\times k}} 
\newcommand{\mar}{\mathcal{A}^{n\times r}}

\DeclareMathOperator*{\argmin}{arg\,min}

\allowdisplaybreaks

\graphicspath{ {./fig/} }

\title{A Sparse Completely Positive Relaxation of the Modularity Maximization for Community Detection} 
\author{Junyu Zhang \thanks{Department of Industrial and System Engineering, University of Minnesota (zhan4393@umn.edu).}
\and
Haoyang Liu \thanks{Beijing International Center for Mathematical Research, Peking University (liuhaoyang@pku.edu.cn).}
\and
Zaiwen Wen
\thanks{Beijing International Center for Mathematical
Research, Peking University (wenzw@pku.edu.cn).
Research supported in part by the NSFC grant 11421101, and by the National
Basic Research Project under the grant 2015CB856002.}
\and
Shuzhong Zhang
\thanks{Department of Industrial and System Engineering, University of Minnesota (zhangs@umn.edu).}
}

\begin{document}
\maketitle 
\begin{abstract}
	In this paper, we consider the community detection problem under either the stochastic block model (SBM) assumption or the degree-correlated
    stochastic block model (DCSBM) assumption. The modularity maximization formulation for the community detection problem is NP-hard in general. In this paper, we propose a sparse and low-rank completely
    positive relaxation for the modularity maximization problem, we then develop an
    efficient row-by-row (RBR) type block coordinate descent (BCD) algorithm to solve
    the relaxation and prove an $\mathcal{O}(1/\sqrt{N})$ convergence rate
    to a stationary point where $N$ is the number of iterations. A fast rounding scheme is constructed to retrieve the
    community structure from the solution. Non-asymptotic high probability
    bounds on the misclassification rate are established to justify our
    approach. We further develop an asynchronous parallel RBR algorithm to speed
    up the convergence. Extensive numerical experiments on both synthetic and
    real world networks show that the proposed approach enjoys advantages in both
    clustering accuracy and numerical efficiency.  Our numerical results indicate that the newly proposed method is a quite competitive alternative for community detection on sparse networks with over 50 million nodes.
\end{abstract}  

\begin{keywords}Community Detection, Degree-Correlated Stochastic Block Model,
  Completely Positive Relaxation, Proximal Block Coordinate Descent Method, Non-asymptotic Error Bound.
\end{keywords}
\begin{AMS}90C10, 90C26, 90C30 \end{AMS}

  \section{Introduction} \label{sec:intro1}
The community detection problem aims to retrieve the underlying
community/cluster structure of a network from the observed nodes connection
data. This network structure inference technique has been reinvigorated because of the modern applications of  large scale networks, including social networks, energy distribution networks and economic networks, etc. The standard frameworks for studying
community detection in networks are the stochastic block model (SBM)
\cite{SBM-1983} and the degree-correlated stochastic block model (DCSBM)
\cite{DCSBM-Newman-2011,DCSBM-Spectral-Dasgupta-2004}. Both of them are random
graph models based on an underlying disjoint community structure. To solve the
community detection problem, there have been a variety of greedy methods which
aim to maximize some quality function of the community structure; see
\cite{otherMethods-Wang-2015} and the references therein. Although these methods are
numerically efficient, their theoretical analysis is still largely missing.

An important class of methods for community detection are the so-called spectral
clustering methods, which exploit either the spectral decomposition of the adjacency matrix \cite{Spectral-Mcsherry-2001, Spectral-Sussman-2012, Spectral-SCORE-Jin-2015}, or the normalized adjacency or graph Laplacian matrix \cite{DCSBM-Spectral-Dasgupta-2004,	Spectral-Coja-2009, 	Spectral-Gulikers-2015, Spectral-Rohe-2011, Spectral-Chaudhuri-2012, Spectral-Qin-2013}. A significant advantage of the spectral clustering methods is that they are numerically efficient and  scalable. Although many of them are shown to be statistically consistent with the SBM and the DCSBM conditions under certain conditions, their numerical performance on synthetic and real data does not fully support such hypothesis. This can also be observed in the numerical experiments where the spectral methods SCORE \cite{Spectral-SCORE-Jin-2015} and OCCAM \cite{Spectral-OCCAM-Zhang-2014} are tested. It is worth mentioning that spectral clustering methods may fail to detect the communities for large sparse networks.

Another closely connected class of methods are based on the nonnegative matrix factorization (NMF) methods. 
In \cite{NMF-Equivalence-Ding-2005}, an interesting equivalence is shown between
some specific NMF-based methods, the kernel K-means method and a specific
spectral clustering method. There have been extensive research on graph
clustering and community detection problems; see e.g. \cite{NMF-Orthoganal-Ding-2006,
NMF-Kim-2008, NMF-Wang-2011, NMF-Psorakis-2011, NMF-Kuang-2012,	NMF-Yang-2012}.
The clustering performance and scalability of these methods are well appreciated, while their theoretical guarantees are still not fully understood. In \cite{NMF-Orthoganal-Paul-2016}, the authors address this issue by showing the asymptotic consistency of their NMF-based community detection algorithm under the SBM or the DCSBM. However, a non-asymptotic error bound is still unknown. 

In addition to the nonconvex methods, significant advances have also been made
by means of  convex relaxation. In
\cite{LR+Sparse-Oymak-2011,Convex-Chen-2012,Convex-Chen-2014}, the authors
propose to decompose the adjacency matrix or its variants into a ``low-rank +
sparse'' form by using nuclear norm minimization methods. In
\cite{CMM-X.Li-2015}, an SDP relaxation of the modularity maximization followed
by a doubly weighted k-median clustering is proposed. This method is
demonstrated to be very competitive with nice theoretical guarantees under the
SBM or the DCSBM against the state-of-the-art methods in various datasets; see
\cite{Convex-Cai+Li-2015,Convex-GrothIEQ-Guedon-2016} for more examples. Although these convex methods are numerically stable and statistically efficient under various model settings, they are algorithmically not scalable. This is because solving these problems  involves expensive subroutines such as full eigenvalue decomposition or matrix inversion, making these algorithms unsuitable for huge datasets.

In this paper, a sparse and low-rank completely positive relaxation for the
modularity maximization problem \cite{Modularity-Newman-2006} is proposed. Specifically, we define an assignment matrix $\Phi$ for $k$ communities among $n$
nodes as an $n\times k$ binary matrix such that each row of $\Phi$ corresponds
to a particular node in the network and contains exactly one element equal to one. The column index of each element indicates which community this node
belongs to. Consequently, we reformulate the modularity maximization as an optimization
over the set of low-rank assignment matrices. Then, we relax the binary
constraints for each row of the assignment matrix and impose the nonnegative and
spherical constraints. The cardinality constraints can be optionally added in
order to handle large scale datasets. To solve the resulting relaxation, we
regard each row of the relaxed assignment matrix as a block-variable and propose a
row-by-row (RBR) proximal block coordinate descent approach. We show that a closed-form
solution for these nonconvex subproblems is available and an
$\mathcal{O}(1/\sqrt{N})$ convergence rate to a stationary point can be obtained.
One can then proceed with either K-means or K-median clustering to retrieve the
community structure from the solution of the relaxed modularity maximization
problem. Since these clustering schemes themselves are expensive, a simple yet  
efficient rounding scheme is constructed for the community retrieval purpose.
Non-asymptotic high-probability bounds are established for the misclassification
rate under the DCSBM. The properites of the SBM can be obtained as a special
case of the DCSBM. We further propose an asynchronous parallel computing scheme
as well as an iterative rounding scheme to enhance the efficiency of our
algorithm. Finally, numerical experiments on both real world and synthetic
networks show that  our method is competitive with the SCORE
\cite{Spectral-SCORE-Jin-2015}, OCCAM \cite{Spectral-OCCAM-Zhang-2014} and CMM
\cite{CMM-X.Li-2015} algorithms regarding both numerical efficiency and
clustering accuracy. We also perform experiments on extremely large sparse
networks with up to 50 million nodes and compare the proposed method with the
LOUVAIN algorithm \cite{2008JSMTE..10..008B}. It turns out that our method is at least
comparable to  the LOUVAIN
algorithm in terms of  finding a
community assignment with a small number of communities and a high modularity. 

The rest of the paper is organized as follows. In Section \ref{section: PrPr},
we set up the community detection problem and introduce the proposed framework.
In Section \ref{section:Algorithm}, we present our algorithm and the convergence
results as well as the asynchronous parallel scheme.  In Section
\ref{section:StatTheories}, we validate the statistical consistency of our
methods by non-asymptotically bounding the misclassification rate under the
DCSBM. Finally, the numerical experiments and comparisons are presented in Section \ref{section:numericalEXP}.

\section{Model Descriptions}\label{section: PrPr}
\subsection{The Degree-Correlated Stochastic Block Model}
In \cite{SBM-1983}, Holland et al. proposed the so called stochastic block model
(SBM), where a non-overlapping true community structure is assumed. This
structure is a partition $\{C_1^*,...,C_k^*\}$ of the nodes set $[n] =
\{1,...,n\}$, that is 
$$C_a^*\cap C_b^* = \emptyset,\forall a\neq b \text{ and } \cup_{a=1}^kC_a^* = [n].$$
Under the SBM, the network is generated as a random graph  characterized by a
symmetric matrix $B\in S^{k\times k}$ with all entries ranging between 0 and 1.
Let $A\in\{0,1\}^{n\times n}$ be the adjacency matrix of the network with
$A_{ii}=0, \forall i\in[n]$. Then for $i\in C_a^*, j\in C_b^*, i\neq j$,
$$A_{ij} = \begin{cases}
1, & \mbox{ with probability }  B_{ab}, \\
0, & \mbox{ with probability }  1-B_{ab}.
\end{cases}$$

One significant drawback of the SBM is that it oversimplifies the network
structure. All nodes in a same community are homogeneously characterized. The
DCSBM is then proposed to alleviate this drawback by introducing the \emph{degree
heterogeneity} parameters $\theta = (\theta_1,...,\theta_n)$ for each node; see
\cite{DCSBM-Spectral-Dasgupta-2004,DCSBM-Newman-2011}. If one keeps the definitions of $B$ and $C_a^*, a=1,...,k$ and still let $A$ be the adjacency matrix of the network with $A_{ii}=0, \forall i\in[n]$, then for $i\in C_a^*, j\in C_b^*, i\neq j$,
$$A_{ij} = \begin{cases}
1, & \mbox{ with probability }  B_{ab}\theta_i\theta_j, \\
0, & \mbox{ with probability }  1-B_{ab}\theta_i\theta_j.
\end{cases}$$ 
Hence, the heterogeneity of nodes  is characterized by this vector $\theta$. Potentially, a node $i$ with larger $\theta_i$ has more edges linking to the other nodes.

\subsection{A Sparse and Low-Rank Completely Positive Relaxation}\label{sec:intro}
For the community detection problem,
one popular method is to maximize a community quality function called ``modularity'' proposed in \cite{Modularity-Newman-2006}. We say that a binary matrix $X$ is a partition matrix, if there exists a partition $\{S_1,...,S_r\}, r\geq 2$ of the index/node set $[n]$ such that $X_{ij} = 1$ for $i,j\in S_t, t\in\{1,...,r\}$ and $X_{ij} = 0$ otherwise. Then a given community structure of a network, not necessarily the true underlying structure, is fully characterized by its associated partition matrix.
We denote the degree vector by $d$ where $d_i = \sum_j A_{ij}, i\in[n].$ Define the matrix \be \label{eq:defC}C = -(A-\lambda dd^{\top}),\ee
where $\lambda = 
1/{\|d\|_1}$. Therefore the modularity of a community structure is defined as $Q = -\iprod{C}{X}.$
Let the set of all partition matrices of $n$ nodes with \emph{no more than} $r$ subsets be denoted by $\mathcal{P}_n^r$. Then the modularity maximization model is formulated as
\be \label{prob:ptn}\min_{X} \iprod{C}{X} \st 
X\in \mathcal{P}_n^n.  \ee
Assume that we know the number of communities in the network is $k$. Then the problem becomes 
\be \label{prob:ptk}\min_{X} \iprod{C}{X} \st 
X  \in\mathcal{P}_n^k. \ee 

Although modularity optimization enjoys many desirable properties, it is an NP-hard problem. A natural relaxation is the convex SDP relaxation considered in \cite{CMM-X.Li-2015}: 
\be \label{prob:SDP}  
\begin{aligned} \min_{X\in \mathbb{R}^{n\times n}} \quad &   \iprod{C}{X} \\ 
	\st &  X_{ii} = 1, i=1,\ldots,n,\\
	& 0\leq X_{ij} \leq 1, \forall i,j,\\
	& X \succeq 0.
\end{aligned}
\ee
This formulation exhibits desirable theoretical properties under the SBM or the
DCSBM. However, the computational cost is high when the network is
large. This paper aims to formulating a model that retains nice theoretical properties while maintaining computationally viable. Recall the definition of an assignment
matrix in Section \ref{sec:intro1}  and suppose that the number of communities $k$ is no more than $r$. Then the true partition matrix $X^*$ can be decomposed as $X^* = \Phi^*(\Phi^*)^{\top}$, where $\Phi^*\in\{0,1\}^{n\times r}$ is the true assignment matrix. Let us define the set of assignment matrix in $\R^{n\times r}$ to be $\mar$, and define the notations 
$$U = \left(u_1,...,u_n\right)^{\top},\quad \Phi = \left(\phi_1,...,\phi_n\right)^{\top},$$
where the $u_i$'s and $\phi_i$'s are $r$-by-$1$ column vectors. By setting $r = n$ and $r = k$, respectively, problem (\ref{prob:ptn}) and (\ref{prob:ptk}) can be formulated as  

\be \label{prob:pt_CP_nk} 
\min_{\Phi} \quad \iprod{C}{\Phi\Phi^{\top}} ~~~~\st~~~~\Phi\in\mar. 
\ee

A matrix $X \in \R^{n\times n}$ is completely positive (CP) if there exists
$ U\in
\R^{n\times r}, r\geq 1$ such
that $ U \geq 0$, $X=  U  U^\top$. 
Let the set of all $n\times n$ CP matrices be denoted by $\cpn$.  Due to the
reformulation (\ref{prob:pt_CP_nk}), by constraining the $\mathcal{CP}$ rank of
$X$,  we naturally obtain a rank-constrained $\mathcal{CP}$ relaxation of problem (\ref{prob:ptn}) and (\ref{prob:ptk}) as
\be \label{prob:CPn}  \begin{aligned} \min_{X} \quad &   \iprod{C}{X} \\ 
	\st &  X_{ii} = 1, i=1,\ldots,n,\\
	& X \in \cpn, \\
	& {\rm rank}_{\mathcal{CP}}(X) \leq r,
\end{aligned}
\ee 
where ${\rm rank}_{\mathcal{CP}}(X): = \min\{s: X = UU^{\top}, U\in\mathbb{R}^{n\times s}_+\}$. Note that the above problem is still  NP-hard. 
By using a decompsition $X=UU^\top$, we propose  an alternative formulation:
\be \label{prob:CP_DCk}  \begin{aligned} \min_{U\in \mathbb{R}^{n\times r}} \quad &   \iprod{C}{UU^{\top}} \\ 
	\st &  \|u_{i}\|^2 = 1, i=1,\ldots,n,\\
	& \|u_{i}\|_0 \leq p, i=1,\ldots,n,\\
	& U \geq 0,
\end{aligned}
\ee 
where the parameter $\lambda$ in $C$ can be set to a different value other than
$1/\|d\|_1$ and $1\leq p \leq r$ is an integer. Note that the additional
$\ell_0$ constraint is added to each row of $U$ to impose sparsity of the
solution. When $p = r$, that is, the $\ell_0$ constraints vanish, this problem
is exactly equivalent to the $\mathcal{CP}$ relaxation \eqref{prob:CPn}.
Usually, the value of $p$ is set to a small number such that the total amount of
the storage of $U$ is affordable on huge scale datasets. 
 Even though \cref{prob:CP_DCk} is still NP-hard, the decomposition and the
 structure of the problem enable us to develop a computationally efficient
 method to reach at least a stationary point.

We summarize the relationships between different formulations as 
\begin{align*}
&(\ref{prob:ptk}) \qquad \qquad \qquad \qquad \qquad  \\
&~\Updownarrow ~\quad \Rightarrow ~ (\ref{prob:CP_DCk})~ \Rightarrow (\ref{prob:CPn})~ \Rightarrow  ~ \eqref{prob:SDP}, \\
&(\ref{prob:pt_CP_nk}) \qquad \qquad   \qquad \qquad \qquad  
\end{align*}
where formulations (\ref{prob:ptk}) and (\ref{prob:pt_CP_nk}) are equivalent
forms of the NP-hard modularity maximization and problem (\ref{prob:SDP}) is the
SDP relaxation. 
The ``$\Rightarrow$'' indicates that our relaxation is tighter than the convex SDP relaxation proposed in \cite{CMM-X.Li-2015}. If $U$ is feasible to our problem (\ref{prob:CP_DCk}), then $UU^{\top}$ is  automatically feasible to the SDP relaxation (\ref{prob:SDP}). 

 
\subsection{A Community Detection Framework} \label{sec:rounding}
Suppose that a solution $U$ is computed from problem (\ref{prob:CP_DCk}), we can
then use one of the following three possible ways to retrieve the community structure from $U$. 
\begin{itemize}
	\item[i.] Apply the K-means clustering algorithm  directly to the rows of $U$ by solving 
	\bea\label{prob:k-means}
	& \min & \|\Phi X_c-U^*\|_F^2 \nonumber\\
	& \st & \Phi\in\mak, X_c\in\R^{k\times k},
	\eea
	where the variable $\Phi$ contains the community/cluster assignment and $X_c$ contains the corresponding centers.
	\item[ii.] Apply the weighted K-means clustering over the rows of $U$, namely, we solve,
	\bea
	\label{prob:w-k-means}
	& \min &\|D(\Phi X_c-U)\|_F^2 \nonumber \\
	& \st & \Phi\in\mak, X_c\in\R^{k\times k},  
	\eea
	where $D = {\rm diag}\{d_1,...,d_n\}$, $\Phi$ and $X_c$ are interpreted the same as that in (i). 
	\item[iii.] Apply a direct rounding scheme to generate the final community assignment $\Phi$. Let $\phi_{i}^{\top}$ be the $i$th row of $\Phi$, 
	\be
	\label{rounding}
	\phi_{i}\leftarrow e_{j_0}, ~~ j_0 = \arg\max_{j} (u_i)_j,
	\ee
	where $e_{j_0}$ is the $j_0$th unit vector in $\mathbb{R}^k$.
\end{itemize} 
All three methods have theoretical guarantees under the SBM or the DCSBM as will be shown in later sections. The overall numerical framework in this paper is as follows:
\begin{itemize}
	\item[Step 1.] Solve problem (\ref{prob:CP_DCk}) starting from a randomly generated
      initial point  
       and obtain an approximate solution.
	\item[Step 2.] Recover the community assignment via one of the three clustering methods.
	\item[Step 3.] Repeat steps 1 and 2 for a couple of times and then pick up a community assignment with the highest modularity.
\end{itemize}

\section{An Asynchrounous Proximal RBR Algorithm}\label{section:Algorithm}
\subsection{A Nonconvex Proximal RBR Algorithm} In this subsection, we setup a
basic proximal RBR algorithm without parallelization for \eqref{prob:CP_DCk}. Let $r$ be chosen as the community number $k$. Define the feasible set for the block variable $u_i$ to be 
$\cU_i = \mathcal{S}^{k-1}\cap\R^{k}_+\cap M_p^k, $
where $\mathcal{S}^{k-1}$ is the $k-1$ dimensional sphere and $M_p^k = \{x\in\R^k:\|x\|_0\leq p\}$.
Define $\cU = \cU_1\times\cdots\times\cU_n.$ Then problem
\eqref{prob:CP_DCk} can be rewritten as 
\be \label{prob:CPSP}   \begin{aligned} \min_{U \in \cU} \quad & f(U)\equiv \iprod{C}{U U^\top}.
\end{aligned}
\ee
For the $i$th subproblem, we fix all except the $i$th row of $U$ and solve the
subproblem
$$u_i = \arg\min_{x\in\cU_i}f(u_1,...,u_{i-1},x,u_{i+1},...,u_n)+\frac{\sigma}{2}\|x-\bar{u}_i\|^2,$$
where $\bar{u}_i$ is some reference point and $\sigma>0$. Note that the spherical constraint eliminates the second order term $\|x\|^2$ and simplifies the subproblem to 
\be
\label{prob:sphere}
u_i = \arg\min_{x\in\cU_i} b^{\top}x,
\ee
where $b = 2C_{-i}^iU_{-i}-\sigma \bar{u}_i$, and $C_{-i}^i$ is the $i$th row of
$C$ without the $i$th component, $U_{-i}$ is the matrix $U$ without the $i$th
row $u_i^{\top}$. For any positive integer $p$, a closed-form solution can be derived
in the next lemma.  
\begin{lemma}
	For problem $u = \argmin \{b^{\top}x :x\in\mathcal{S}^{k-1}\cap \R^k_+\cap M_p^k\}$, define $b^+ = \max\{b,0\},b^- = \max\{-b,0\},$ where the $\max$ is taken component-wisely. Then the closed-form solution is given by 
	\be u  = \begin{cases} 
	\frac{b^-_p}{\|b^-_p\|}, & \mbox{ if } b^-\neq 0,
		\\
		e_{j_0}, \mbox{ with } j_0 = \arg\min_j b_j, & \mbox{ otherwise},
	\end{cases}\ee
    where $b^-_p$ is obtained by keeping the top $p$ components in $b^-$ and
    letting the others be zero, and when $\|b^-\|_0\leq p$, $b^-_p = b^-$.
\end{lemma}
\begin{proof} When $b^-\neq0$, for $\forall u\in\R^k_+, \|u\|^2 = 1, \|u\|_0\leq p$, we obtain
\begin{eqnarray*}
	b^{\top}u & = & (b^+-b^-)^{\top}u \geq  -(b^-)^{\top}u  \geq -\|b_p^-\|,
\end{eqnarray*}
where the equality is achieved at $u = \frac{b_p^-}{\|b_p^-\|}$. 

When $b^-=0$, i.e., $b\geq0,$ we have
\begin{eqnarray}
b^{\top}u = \sum_i b_iu_i \geq \sum_i b_iu_i^2 \geq  b_{\min}\sum_i u_i^2 = b_{\min}, \nonumber
\end{eqnarray}
where $b_{\min} = \min_j b_j$, and the equality is achieved at $u = e_{j_0}, j_0 =
\arg\min_j b_j.$ \end{proof}

Our proximal RBR algorithm is outlined in Algorithm \ref{alg:RBR}.

\begin{algorithm2e}[H]\label{alg:RBR} \caption{A Row-By-Row (RBR) Algorithm} 
	Give $U^0$, $\sigma>0$. Set $t = 0$.\\
	\While{Not converging }{
		\For{i = 1,...,n}{
		$u_i^{t+1} = \arg\min_{x\in\cU_i}f(u_1^{t+1},...,u_{i-1}^{t+1},x,u_{i+1}^t,...,u_n^t)+\frac{\sigma}{2}\|x-u_i^t\|^2$.\\	
    	}
        $t \leftarrow t+1$.
	}
\end{algorithm2e}

Note that the objective function is smooth with a Lipschitz continuous gradient. When $p=k$, the $\ell_0$ constraints hold trivially and thus can be removed. Then problem \eqref{prob:CP_DCk} becomes a smooth nonlinear program. Therefore, the KKT condition can be written as
$$\begin{cases}
 2CU^*+2\Lambda U^{*}-V = 0,\\
\|u_i^{*}\| = 1, ~~~u_i^{*}\geq 0 , ~~~~ v_i\geq 0, & \mbox{ for } i = 1,...,n,\\ 
(u_i^{*})_j(v_i)_j = 0, & \mbox{ for } i = 1,...,n \mbox{ and } j = 1,...,k,
\end{cases}$$
where $\Lambda = {\rm diag}\{\lambda\}, V = \left(v_1,...,v_n\right)^{\top}$ and
$\lambda_i\in\R, v_i\in\R^k_+$ are the Lagrange multipliers associated with the
constraints $\|u_i\|^2=1, u_i\geq0$, respectively.  Then we have the following convergence result.

\begin{theorem}
	Suppose that the sequence $\{U^0,U^1,...,U^N\}$ is generated by Algorithm \ref{alg:RBR} with a chosen proximal parameter $\sigma>0$ and $p = k$. Let 
	\be
	\label{thm:convergence-0}
	t^* := \arg\min_{1\leq t\leq N}\bigg\{\|U^{t-1}-U^t\|_F^2\bigg\}.
	\ee
	Then there exist Lagrange multipliers $\lambda_i\in\R, v_i\in\R^k_+, i =
    1,...,n$, such that 
	$$\begin{cases}
	\big\|2CU^{t^*}+2\Lambda U^{t^*}-V\big\|_F\leq 2\sqrt{\frac{2\|C\|_1(4\|C\|_F^2+\sigma^2)}{N\sigma}},\\
	\|u_i^{t^*}\| = 1, ~~~u_i^{t^*}\geq 0 , ~~~~ v_i\geq 0, & \mbox{ for } i = 1,...,n,\\ 
	(u_i^{t^*})_j(v_i)_j = 0, & \mbox{ for } i = 1,...,n \mbox{ and } j = 1,...,k,
	\end{cases}$$
	where $\|C\|_1$ is the component-wise $\ell_1$ norm of $C$.
\end{theorem}
\begin{proof} The proof consists mainly of two steps.  First, we show that $\sum_{t=1}^N\|U^t-U^{t-1}\|^2_F$ is moderately bounded. By the optimality of the subproblems, we have that for $ i = 1,...,n, $
\begin{equation*}
f(u_1^{t},...,u_{i-1}^{t},u_{i}^{t-1},...,u_n^{t-1})-f(u_1^{t},...,u_{i}^{t},u_{i+1}^{t-1},...,u_n^{t-1}) \geq \frac{\sigma}{2}\|u_{i}^{t}-u_{i}^{t-1}\|^2.
\end{equation*}
Summing these inequalities up over $i$ gives
$$
f(U^{t-1}) - f(U^{t}) \geq \frac{\sigma}{2}\|U^{t}-U^{t-1}\|_F^2,
$$
which yields
\be
\label{thm:convergence-1}
\sum_{t=1}^N\|U^t-U^{t-1}\|^2_F\leq\frac{2}{\sigma}\left(f(U^0)-f(U^N)\right).
\ee
Note that $|\langle u_i,u_j\rangle|\leq \|u_i\|\|u_j\| = 1$, and so for any feasible $U$,
\beaa
f(U) = \langle CU, U\rangle = \sum_{i,j} C_{i,j}\langle u_i,u_j\rangle\in[-\|C\|_1, \|C\|_1].
\eeaa
Combining with \eqref{thm:convergence-1} and the definition of $t^*$ in \eqref{thm:convergence-0}, it holds that
\be
\label{thm:convergence-2}
\|U^{t^*}-U^{t^*-1}\|^2_F\leq\frac{4}{N\sigma}\|C\|_1.
\ee

Second, also by the optimality of the subproblems, there exist KKT multipliers $\lambda_i\in\R$ and $v_i\in\R^k_+$ for each subproblem such that, 
\be
\label{thm:convergence-subkkt}
\begin{cases}
\nabla_if(u_1^{{t^*}},...,u_i^{t^*},u_{i+1}^{{t^*}-1},...,u_n^{{t^*}-1})+\sigma(u_i^{t^*}-u_i^{{t^*}-1})+2\lambda_iu_i^{t^*}-v_i = 0,\\
\|u_i^{t^*}\|^2 = 1, u_i^{t^*}\geq0, v_i\geq0,\\
(u_i^{t^*})_j(v_i)_j = 0, \mbox{ for } j = 1,...,k.
\end{cases}
\ee
Define $\tilde{U}^{{t^*},i} = \left[u_1^{t^*},...u_i^{t^*},u_{i+1}^{{t^*}-1},...,u_n^{{t^*}-1}\right]^{\top}$, then  
\bea
& & \|\nabla_if(U^{t^*})-\nabla_if(u_1^{{t^*}},...,u_i^{t^*},u_{i+1}^{{t^*}-1},...,u_n^{{t^*}-1})\|^2 \nonumber\\
& = & \|2C_i\left(U^{t^*} - \tilde{U}^{{t^*},i}\right)\|^2 \nonumber \\
& \leq & 4\|C_i\|^2\|U^{t^*} - \tilde{U}^{{t^*},i}\|_F^2\\
& \leq & 4\|C_i\|^2\|U^{t^*} - {U}^{{t^*}-1}\|_F^2,\nonumber
\eea
where $C_i$ is the $i$th row of matrix $C$. Consequently,
\bea
\label{thm:convergence-main}
& & \big\|2CU^{t^*}+2\Lambda U^{t^*}-V\big\|_F^2 \nonumber\\
& = & \sum_{i=1}^n \big\|\nabla_if(U^{t^*})+2\lambda_i u_i^{t^*}-v_i\big\|^2 \nonumber\\
& = & \sum_{i=1}^n\big\|\nabla_if(U^{t^*}) -\nabla_if(u_1^{{t^*}},...,u_i^{t^*},u_{i+1}^{{t^*}-1},...,u_n^{{t^*}-1})-\sigma(u_i^{t^*}-u_i^{{t^*}-1})\big\|^2\nonumber\\
& \leq &2\sum_{i=1}^n(\|\nabla_if(U^{t^*})
-\nabla_if(u_1^{{t^*}},...,u_i^{t^*},u_{i+1}^{{t^*}-1},...,u_n^{{t^*}-1})\|^2+\sigma^2\|u_i^{t^*}-u_i^{{t^*}-1}\|^2)\\
& \leq & 2(4\|C\|_F^2+ \sigma^2)\|U^{t^*}-U^{t^*-1}\|_F^2 \nonumber\\
& \leq & \frac{8\|C\|_1(4\|C\|_F^2+\sigma^2)}{T\sigma},\nonumber
\eea
where the second equality is due to \eqref{thm:convergence-subkkt} and the last inequality is due to \eqref{thm:convergence-2}. 
Combining the bounds in \eqref{thm:convergence-main} and \eqref{thm:convergence-subkkt} proves the theorem.
\end{proof}

\subsection{An Asynchronous Parallel Proximal RBR Scheme}
In this subsection, we briefly discuss the parallelization of the RBR Algorithm \ref{alg:RBR}. 
The overall parallel setup for the algorithm is a shared memory model with many threads. The variable $U$ is in the shared memory so that it can be accessed by all threads. The memory locking is not imposed throughout the process. Namely, even when a thread is updating some row $u_i$ of $U$, the other threads can still access $U$ whenever needed.


Before explaining the parallel implementation, let us first make clear the sequential case. The main computational cost of updating one row $u_i$ by solving the subproblem (\ref{prob:sphere}) is due to the computation of $b =
2C_{-i}^iU_{-i}-\sigma\bar{u}_i$, where $\bar{u}_i$ and $U$ are the current iterates. Note that the definition of $C$ in \eqref{eq:defC} yields
\be \label{eq:b-update} b^{\top} = -2A_{-i}^iU_{-i}+2\lambda d_id_{-i}^{\top}U_{-i}-\sigma\bar{u}_i,\ee
where $A_{-i}^i$ is the $i$th row of $A$ without the $i$th component. By the
sparsity of $A$, computing $-A_{-i}^iU_{-i}$ takes $\mathcal{O}(d_ip)$ flops. As
for the term $\lambda d_id_{-i}^{\top}U_{-i}$, we can compute the vector $d^{\top}U =
\sum_{i = 1}^{|V|}d_iu_i^{\top}$ once, then store it throughout all iterations. For each time after a row $u_i$ is updated, we can evaluate this vector at a cost of
$\mathcal{O}(p)$ flops. Hence, the total computational cost of one full sweep of all rows is $\mathcal{O}(2(k+p)|V|+p|E|)$.

 Our parallel implementation is outlined in  Algorithm \ref{alg:pRBR} where many
 threads working at the same time. The vector $d^{\top}U$ and matrix $U$ are
 stored in the \emph{shared} memory and they can be accessed and updated by all
 threads. Each thread picks up one row $u_i$ at a time and then it reads $U$ and
 the vector $d^{\top}U$. Then an individual copy of the vector $b^{\top}$ is
 calculated, i.e., $b^{\top}$ is owned \emph{privately} and cannot be accessed
 by other threads. Thereafter, the variable  $u_i$ is updated and $d^{\top}U$ is
 set  as $d^{\top}U \leftarrow d^{\top}U + d_i(u_i-\bar{u}_i)$ in the shared
 memory. Immediately without waiting for other threads to finish their
 computation, it proceeds to another row. Hence, unlike the situation in
 Algorithm \ref{alg:RBR},  other blocks of variables $u_j,j\neq i$ are not
 necessarily up to date when a thread is updating a row $u_i$ and $d^{\top}U$.
 Moreover, if  another thread is just modifying some row $u_j$ or the vector
 $d^{\top}U$ when this thread is reading these memories, it will make the update
 of $u_i$ with the partially updated data. A possible way to avoid the partially
 updated conflict is to add memory locking when a thread is changing the
 memories of $U$ and $d^{\top}U$. In this way,   other threads will not be able
 to access these memories and will wait until the update completes. However, this process may cause many cores to be idle so that it only provides limited speedups or even slows the algorithm down. In our implementation, these memory locking are completely removed.  Consequently, our method may be able to provide near-linear speedups.

In order to increase the chance of finding a good solution, we perform the
simple  rounding scheme \eqref{rounding} in Section \ref{sec:rounding} whenever it is
necessary.  
By applying this technique, we obtain a binary solution after each
iteration as a valid community partition. Although the convergence is not guaranteed in this case,
it is numerically robust throughout our experiments. Note that the rounding technique can also be asynchronously parallelized because the update of all rows are completely independent. 

It is worth noting that our algorithm is similar to the HOGWILD! algorithm \cite{AsynParallel-Hogwild-2011}, where HOGWILD! minimizes the objective function in a stochastic gradient descend style while we adopt the exact block coordinate minimization style. There are no conflicts in our method when updating the variable $U$ compared to the HOGWILD!, since each row can only be modified by up to one thread. However, the issue of conflicts arises when updating the intermediate vector $d^{\top}U$. Another similar asynchronous parallel framework is CYCLADES \cite{AsynParallel-Cyclades-2016}. Consider a ``variable conflict graph'', of which the nodes are grouped as blocks of variables and an edge between two blocks of variables exists if they are coupled in the objective function. The CYCLADES carefully utilizes the sparsity of the variable conflict graph and designs a particular parallel updating order so that no conflict emerges due to the asynchronism with high probability. This enables the algorithm to be asynchronously implemented while keeping the sequence of iterates equivalent to that generated by a sequential algorithm. However, because the block variables are all coupled together in formulation \eqref{prob:CP_DCk}, no variable sparsity can be exploited in our case. 

\begin{algorithm2e}[H]\caption{Asynchronous parallel RBR algorithm}
	\label{alg:pRBR}
	Give $U^0$, set $t = 0$\\
	\While{Not converging }{
		\For{each row $i$ asynchronously}{
			Compute the vector $b_i^{\top} = -2A_{-i}^iU_{-i}+2\lambda
			d_id_{-i}^{\top}U_{-i}-\sigma u_i$, and save previous iterate  $\bar{u}_i$ in the private memory.\\
			Update   $u_i \leftarrow
			\argmin_{x\in\mathcal{U}_i}b_i^{\top}x$ in the shared memory. \\
			Update the vector $d^{\top}U \leftarrow d^{\top}U + d_i(u_i-\bar{u}_i)$ in the shared memory.
		}
		\If{rounding is activated}{
			\For{each row $i$ asynchronously}{
				Set $u_i=e_{j_0}$ where $j_0=\arg\max (u_i)_j$.\\
			}
			Compute and update $d^{\top}U$.
		}
	}
\end{algorithm2e}

\section{Theoretical Error Bounds Under The DCSBM}\label{section:StatTheories}
In this section, we establish theoretical results of our method. Throughout the
discussion, we focus mainly on the DCSBM, as these assumptions and results for
the SBM can be easily derived from those of the DCSBM as special cases. Define 
$$G_a = \sum_{i\in C_a^*}\theta_i,~~~~ J_a = \sum_{i\in C_a^*}\theta_i^2,~~~~M_a = \sum_{i\in C_a^*}\theta_i^3,  ~~~~ n_a = |C_a^*|,$$
$$H_a = \sum_{b=1}^kB_{ab}G_b,\text{ and } f_i = H_a\theta_i, \text{ given } i\in C_a^*.$$
A direct computation gives 
$\mathbb{E}d_i = \theta_iH_a-\theta_i^2B_{aa}\approx f_i.$ 
For the ease of notation, we define $n_{\max} = \max_{1\leq a\leq k}n_a$, and
define $n_{\min}, H_{\max},H_{\min}, J_{\max}, J_{\min}, M_{\max}, M_{\min}$ in a
similar fashion. We adopt the \emph{density gap} assumption in
\cite{CMM-X.Li-2015} to guarantee the model's identifiability.
\begin{assumption}\label{assumption:densityGap} (Chen, Li and Xu, 2016,\cite{CMM-X.Li-2015})
Under the DCSBM, the density gap condition holds that 
	$$\max_{1\leq a<b\leq k}\frac{B_{ab}}{H_aH_b}<\min_{1\leq a\leq k}\frac{B_{aa}}{H_a^2}.$$	
\end{assumption}


Define 
$\|X\|_{1,\theta} = \sum_{ij}|X_{ij}|\theta_i\theta_j$ as the weighted $\ell_1$ norm. Then we have the following lemma which bounds the difference between the approximate partition matrix $U^*(U^*)^{\top}$ and the true partition matrix $\Phi^*(\Phi^*)^{\top}$, where $U^*$ is the optimal solution to problem \eqref{prob:CP_DCk}.

\begin{lemma}\label{lemma:cmm-translate}
	Suppose that Assumption \ref{assumption:densityGap} holds and the tuning parameter $\lambda$ satisfies 
	\be
	\label{lm:cmm-translate-1}
	\max_{1\leq a<b\leq k}\frac{B_{ab}+\delta}{H_aH_b}<\lambda<\min_{1\leq a\leq k}\frac{B_{aa}-\delta}{H_a^2}
	\ee
	for some $\delta>0$. Let $U^*$ be the global optimal solution to problem
    (\ref{prob:CP_DCk}), and define $\Delta = U^*(U^*)^{\top}-\Phi^*(\Phi^*)^{\top}$. Then
    with probability at least $0.99-2(e/2)^{-2n}$, we have 
	$$\|\Delta\|_{1,\theta}\leq \frac{C_0}{\delta}\left(1+\left(\max_{1\leq a\leq k}\frac{B_{aa}}{H_a^2}\|f\|_1\right)\right)(\sqrt{n\|f\|_1}+n),$$
	where $C_0>0$ is some absolute constant that does not depend on problem
    scale and parameter selections.
\end{lemma}

\begin{proof} Lemma \ref{lemma:cmm-translate} is proved directly from Theorem 1
in \cite{CMM-X.Li-2015} for the convex SDP relaxation \eqref{prob:SDP}. Here we
only show the main steps of the proof  and how they can be translated from the original result of  Chen, Li and Xu in \cite{CMM-X.Li-2015}.  By the optimality of $U^*$ and the feasibility of $\Phi^*$ to problem \eqref{prob:CP_DCk}, we have
\bea
\label{lm:cmm-translate-0}
0 & \leq &  \langle  \Delta,A-\lambda dd^{\top}\rangle
 =  \underbrace{\langle\Delta, \mathbb{E}A-\lambda ff^{\top}\rangle}_{S_1}+ \underbrace{\lambda\langle\Delta,ff^{\top}-dd^{\top}\rangle}_{S_2} + \underbrace{\langle\Delta, A-\mathbb{E}A\rangle}_{S_3}.
\eea
The remaining task is to bound the three terms separately.

First, we bound the term $S_1$. Note that $U^*(U^*)^{\top}$ is feasible to the
SDP problem \eqref{prob:SDP}, we still have for $\forall i,j\in C^*_a, a\in[k]$,
$\Delta_{ij}\leq 0$; for $\forall i\in C^*_a, \forall j\in C^*_b, a\neq b,
a,b\in[k], \Delta_{ij}\geq 0$. Combining with \eqref{lm:cmm-translate-1}, the
proof in \cite{CMM-X.Li-2015} is still valid and yields 
$$\Delta_{ij}(\mathbb{E}A-\lambda ff^{\top})_{ij}\leq -\delta\theta_i\theta_j|\Delta_{ij}|.$$
Hence we have 
$$S_1\leq -\delta\|\Delta\|_{1,\theta}.$$
Similarly, the proof goes through for the bounds of $S_2$ and $S_3$. It can be shown that
$$S_2\leq C\left(\max_{1\leq a\leq k}\frac{B_{aa}}{H_a^2}\right)\|f\|_1\left(\sqrt{n\|f\|_1}+n\right)$$
holds with probability at least 0.99, where $C>0$ is some absolute constant. 
Let $K_G \leq 1.783$ be the constant for the original Grothendieck's inequality. Then 
$$S_3\leq 2K_G\sqrt{8n\|f\|_1}+\frac{16K_G}{3}n$$
holds with probability at least $1-2(e/2)^{-2n}$. Combining the bounds for
$S_1,S_2,S_3$ with the inequality \eqref{lm:cmm-translate-0} proves the theorem.
\end{proof}

Lemma \ref{lemma:cmm-translate} indicates that $U^*(U^*)^{\top}$ is close to
$\Phi^*(\Phi^*)^{\top}$. Naturally, we want to bound the difference between the
factors $U^*$ and $\Phi^*$. However, note that for any orthogonal matrix $Q$,
$(U^*Q)(U^*Q)^{\top} = U^*(U^*)^{\top}$. This implies that in general $U^*$ is not
necessarily close to $\Phi^*$ unless multiplied by a proper orthogonal matrix.
To establish this result, we use a matrix perturbation lemma 
in \cite{Book-Matrix-2012}.
\begin{lemma}\label{lemma:perturb}(Corollary 6.3.8, in \cite{Book-Matrix-2012} on page 407)
	Let $A,E\in\mathbb{C}^{n\times n}$. Assume that $A$ is Hermitian and $A+E$ is normal. Let $\lambda_1\leq\lambda_2\leq\cdots\leq\lambda_n$ be the eigenvalues of $A$, and let $\hat{\lambda}_1, \hat{\lambda}_2,...,\hat{\lambda}_n$ be the eigenvalues of $A+E$ arranged in order $Re(\hat{\lambda}_1)\leq Re(\hat{\lambda}_2)\leq\cdots\leq Re(\hat{\lambda}_n)$. Then 
	$$\sum_{i = 1}^n\|\lambda_i-\hat{\lambda}_i\|^2\leq\|E\|_F^2,$$
	where $\|\cdot\|_F$ is the matrix Frobenius norm.
\end{lemma}

In the real symmetric case, this lemma states that we can bound the difference between every eigenvalue of $A$ and perturbed matrix $A+E$ if the error term $\|E\|_F^2$ can be properly bounded. 
\begin{theorem}\label{theorem:dcsbm-Frobenius-err}
	Let $\Theta = {\rm diag}\{\theta_1,...,\theta_n\}$, and let $U^*, \Phi^*$, $\Delta$ be defined according to Lemma \ref{lemma:cmm-translate}. Then there exist orthogonal matrices $Q_1^*$ and $Q_2^*$ such that  
	$$\|\Theta(U^*Q_1^*-\Phi^*)\|_F^2\leq2\sqrt{k}\frac{M_{\max}}{J_{\min}}\|\Delta\|_{1,\theta}^{\half},$$
	$$\|U^*Q_2^*-\Phi^*\|_F^2\leq2\sqrt{k}\frac{n_{\max}}{n_{\min}}\|\Delta\|_{1}^{\half}\leq2\sqrt{k}\frac{n_{\max}}{\theta_{\min}n_{\min}}\|\Delta\|_{1,\theta}^{\half}.$$
\end{theorem}
\begin{proof}
(i). We first prove the bound for $Q_1^*$. A direct calculation yields 
\bea
\label{thm:dcsbm-Frobenius-err:nuclear}
v(\theta) & := &  \min_{Q^{\top}Q=I} \|\Theta(U^*Q-\Phi^*)\|_F^2 \nonumber\\
& = & \min_{Q^{\top}Q=I} \|\Theta U^*Q\|_F^2 + \|\Theta\Phi^*\|_F^2 -2\iprod{\Theta U^*Q}{\Theta\Phi^*} \\
& = & 2\sum_{i = 1}^n\theta_i^2 - 2\max_{Q^{\top}Q=I}\iprod{Q}{(U^*)^{\top}\Theta^2\Phi^*} \nonumber \\
& = & 2\left(\sum_{i = 1}^n\theta_i^2-\|(U^*)^{\top}\Theta^2\Phi^*\|_*\right),\nonumber
\eea
where $\|\cdot\|_*$ denotes the nuclear norm of a matrix, and the last equality
is due to the following argument. Note that
\bea
Q_1^* & = & \arg\max_{Q^{\top}Q=I}\iprod{Q}{(U^*)^{\top}\Theta^2\Phi^*} \nonumber\\
& = & \arg\min_{Q^{\top}Q=I}\|Q-(U^*)^{\top}\Theta^2\Phi^*\|_F^2. \nonumber
\eea
Suppose that we have the singular value decomposition $(U^*)^{\top}\Theta^2\Phi^* = R\Sigma P^{\top},$
then $Q_1^* = RP^{\top}$ solves the above nearest orthogonal matrix problem. Hence, 
$$\iprod{Q_1^*}{(U^*)^{\top}\Theta^2\Phi^*} = {\rm Tr}(\Sigma) = \|(U^*)^{\top}\Theta^2\Phi^*\|_*.$$ 
For any matrix $Z$, denote by $\sigma_i(Z)$ its $i$th singular value, and $\lambda_i(Z)$ its $i$th eigenvector. For a positive semidefinite matrix $Z$, $\sigma_i(Z) = \lambda_i(Z)$. Then
\bea
\label{lm:factor_err_1}
\sigma_i^2((U^*)^{\top}\Theta^2\Phi^*) & = & \lambda_i((\Phi^*)^{\top}\Theta^2U^*(U^*)^{\top}\Theta^2\Phi^*) \nonumber\\
& = & \lambda_i((\Phi^*)^{\top}\Theta^2\Phi^*(\Phi^*)^{\top}\Theta^2\Phi^*+(\Phi^*)^{\top}\Theta^{1.5}\Delta_{\Theta}\Theta^{1.5}\Phi^*), 
\eea
where 
$\Delta_{\Theta} = \Theta^{\half}(U^*(U^*)^{\top}-\Phi^*(\Phi^*)^{\top})\Theta^{\half} =
\Theta^{\half}\Delta\Theta^{\half}.$ The theorem is proved by the following
three steps and by applying Lemma \ref{lemma:perturb} to bound every $\sigma_i^2((U^*)^{\top}\Theta^2\Phi^*)$.

\textbf{Step (a)}. We derive a tight estimation for every eigenvalue
$\lambda_i((\Phi^*)^{\top}\Theta^2\Phi^*(\Phi^*)^{\top}\Theta^2\Phi^*),$ $i = 1,...,k$.  Let
$v_a,1\leq a\leq k$, be the $a$th column of $\Phi^*$. Then we have
$$v_a(t) = \begin{cases}
1, & t\in C_a^*, \\
0, & t\notin C_a^*.
\end{cases}$$
Therefore, we obtain
\beaa
((\Phi^*)^{\top}\Theta^2\Phi^*)_{ab}  =  \sum_{t=1}^n\theta_t^2v_a(t)v_b(t)  =  \begin{cases}
0, & a\neq b,\\
J_a, & a = b,
\end{cases}
\eeaa
which yields $$((\Phi^*)^{\top}\Theta^2\Phi^*)^2 = {\rm diag}\{J_1^2,...,J_k^2\} \text{ and } \lambda_i\left(((\Phi^*)^{\top}\Theta^2\Phi^*)^2\right) = J_i^2.$$

\textbf{Step (b).}  We bound the error term $\|(\Phi^*)^{\top}\Theta^{1.5}\Delta_{\Theta}\Theta^{1.5}\Phi^*\|_F^2$ for (\ref{lm:factor_err_1}): 
\be \label{eq:w1}
\|(\Phi^*)^{\top}\Theta^{1.5}\Delta_{\Theta}\Theta^{1.5}\Phi^*\|_F^2 \leq   \|\Delta_{\Theta}\|_F^2\|\Theta^{1.5}\Phi^*(\Phi^*)^{\top}\Theta^{1.5}\|_2^2 
 \leq  \|\Delta\|_{1,\theta}M_{\max}^2, 
\ee
where the second inequality is due to the following argument. 

By the feasibility of $U^*$ and $\Phi^*$, it is not hard to see that 
$\max_{1\leq i,j\leq n}|\Delta_{ij}|\leq 1$. Therefore  
\beaa
\|\Delta_{\Theta}\|_F^2   =   \sum_{1\leq i,j\leq n}\Delta_{ij}^2\theta_i\theta_j 
 \leq   \sum_{1\leq i,j\leq n}|\Delta_{ij}| \theta_i\theta_j  
  =   \|\Delta\|_{1,\theta}. 
\eeaa
By properly permuting the row index, the true assignment matrix $\Phi^*$ can be written as a block diagonal matrix ${\rm diag}\{\mathbf{1}_{n_1\times1},...,\mathbf{1}_{n_k\times1}\}$,
where $\mathbf{1}_{n_i\times1}, 1\leq i\leq k$ is an $n_i\times1$ vector of all
ones. Consequently, $\Phi^*(\Phi^*)^{\top} = {\rm diag}\{\mathbf{1}_{n_1\times
n_1},...,\mathbf{1}_{n_k\times n_k}\}$ is also a block diagonal matrix, where
$\mathbf{1}_{n_i\times n_i}, 1\leq i\leq k$ is an $n_i\times n_i$ all-one square
matrix. The term $\Theta^{1.5}\Phi^*(\Phi^*)^{\top}\Theta^{1.5}$ is also a
block diagonal matrix. Let $\theta_{(a)} = (\theta_{i_1},...,\theta_{i_{n_a}})^{\top}, \Theta_{(a)}={\rm diag}\{\theta_{(a)}\}$, where $\{i_1,...,i_{n_a}\} = C_a^*$. Then the $a$th block of $\Theta^{1.5}\Phi^*(\Phi^*)^{\top}\Theta^{1.5}$ is 
$$\Theta_{(a)}^{1.5}\mathbf{1}\mathbf{1}^{\top}\Theta_{(a)}^{1.5} = \theta_{(a)}^{1.5}(\theta_{(a)}^{1.5})^{\top},$$
whose eigenvalues are $M_a$ and zeros. Hence, we have 
$$\|\Theta^{1.5}\Phi^*(\Phi^*)^{\top}\Theta^{1.5}\|_2^2 = \max_{1\leq a\leq k}M_a^2 = M_{\max}^2.$$
Therefore,  the inequality \eqref{eq:w1} is proved.

\textbf{Step (c).}  We next bound the objective function $v(\theta)$ defined in
(\ref{thm:dcsbm-Frobenius-err:nuclear}). Note that both matrices
$(\Phi^*)^{\top}\Theta^2\Phi^*(\Phi^*)^{\top}\Theta^2\Phi^*$ and
$(\Phi^*)^{\top}\Theta^{1.5}\Delta_{\Theta}\Theta^{1.5}\Phi^*$ are real
symmetric, hence are  Hermitian and normal. We can apply Lemma \ref{lemma:perturb}. For the ease of notation, let the $k$ singular values of $(U^*)^{\top}\Theta^2\Phi^*$ be $\sigma_1,...,\sigma_k$, and define $\delta_i = |\sigma_i^2-J_i^2|$. Lemma \ref{lemma:perturb} implies that 
$$\sum_{i=1}^k\delta_i^2\leq \|(\Phi^*)^{\top}\Theta^{1.5}\Delta_{\Theta}\Theta^{1.5}\Phi^*\|_F^2 \leq \|\Delta\|_{1,\theta}M_{\max}^2.$$
Together with 
\begin{equation}
\label{eq-2333}
\sigma_i \geq \sqrt{J_i^2-\delta_i} = J_i\sqrt{1-\frac{\delta_i}{J_i^2}} \geq J_i(1-\frac{\delta_i}{J_i^2}) = J_i - \frac{\delta_i}{J_i},
\end{equation}
we have
\beaa
v(\theta) & = & 2\left(\sum_{i=1}^n\theta_i^2 - \sum_{i = 1}^k\sigma_i\right) \leq 2\left(\sum_{i=1}^kJ_i - \sum_{i = 1}^k(J_i - \frac{\delta_i}{J_i})\right) \\
& = & 2\sum_{i = 1}^k \frac{\delta_i}{J_i} \leq 2\left(\sum_{i = 1}^k\frac{1}{J^2_i}\right)^{\half}\left(\sum_{i = 1}^k\delta_i^2\right)^{\half} \\
& \leq & 2\sqrt{k}\frac{M_{\max}}{J_{\min}}\|\Delta\|_{1,\theta}^{\half}.
\eeaa

We need to mention that in \eqref{eq-2333}, $\sigma_i>\sqrt{2}J_i>J_i-\frac{\delta_i}{J_i}$ when $J_i^2-\delta_i<0$. Consequently, the first inequality of Theorem \ref{theorem:dcsbm-Frobenius-err} is established.

(ii) Applying the first inequality with $\Theta = I$ gives 
$$\|U^*Q^*_2-\Phi^*\|_F^2\leq 2\sqrt{k}\frac{n_{\max}}{n_{\min}}\|\Delta\|_1^{\half}.$$
Together with $$\|\Delta\|_{1,\theta} = \sum_{ij}|\Delta_{ij}|\theta_i\theta_j\geq\theta_{\min}^2\|\Delta\|_1,$$
we prove the second inequality. \end{proof}


The bound on $\|U^*Q_2^*-\Phi^*\|_F^2$ implies that with a proper orthogonal
transform, $U^*$ is close to the true community assignment matrix $\Phi^*$. Then
if we apply K-means under $\ell_2$ norm metric over the rows of $U^*$, which is
equivalent to clustering over rows of $U^*Q_2^*$, the clustering result will
intuitively be close to the true community structure. Similar intuition applies
to the weighted K-means method (\ref{prob:w-k-means}). Suppose that the solution to the K-means clustering problem (\ref{prob:k-means}) is $\bar{\Phi}$. For an arbitrary permutation $\Pi$ of the columns of $\bar{\Phi}$, we define the set
$$Err_{\Pi}(\bar{\Phi}) = \{i \mid (\bar{\Phi}\Pi)_i\neq\phi^*_i\},$$
where $\Pi$ also stands for the corresponding permutation matrix and
$(\bar{\Phi}\Pi)_i$ denotes the $i$th row of the matrix $\bar{\Phi}\Pi$. This
set contains the indices of the rows of $\bar{\Phi}\Pi$ and $\Phi^*$ that do not
match. Since applying any permutation $\Pi$ to the columns of $\bar{\Phi}$ does not change the cluster structure at all, a fair error measure should be 
$$\min_{\Pi}|Err_{\Pi}(\bar{\Phi})|,$$
where the minimization is performed over the set of all permutation matrix. Under this measure, the following theorem guarantees the quality of our clustering result. 
\begin{theorem}\label{theorem:dcsbm-kmeans}
	Suppose that $\bar{\Phi}$ is generated by some K-means clustering algorithm
    with cluster number $k$ equal to the number of communities. Then under the DCSBM,
	$$\min_{\Pi}|Err_{\Pi}(\bar{\Phi})|\leq \frac{8\sqrt{k}n_{\max}}{\theta_{\min}n_{\min}}\left(1+\frac{n_{\max}}{n_{\min}}\right)(1+\alpha(k))\|\Delta\|_{1,\theta}^{\half},$$
	where $\alpha(k)$ is the approximation ratio of the polynomial time K-means approximation algorithm, and $\|\Delta\|_{1,\theta}$ is bounded with high probability by Lemma \ref{lemma:cmm-translate}.
\end{theorem} 

\emph{Proof.} Consider the K-means problem (\ref{prob:k-means}). Suppose that the solution generated by some polynomial time approximation algorithm is denoted by $\bar{\Phi},\bar{X}_c$, and the global optimal solution is denoted by $\Phi^{opt}, X_c^{opt}$. Suppose that the approximation ratio of the algorithm is $\alpha(k)$, where $k$ is the cluster number in the algorithm. Then directly
$$\|\bar{\Phi}\bar{X}_c-U^*\|_F^2\leq \alpha(k)\|\Phi^{opt}X_c^{opt}-U^*\|_F^2.$$
By the optimality of $\Phi^{opt}, X_c^{opt}$ and the feasibility of $\Phi^*, (Q_2^*)^{\top}$, where $Q_2^*$ is defined in Theorem \ref{theorem:dcsbm-Frobenius-err}, we have 
$$\|\Phi^{opt}X_c^{opt}-U^*\|_F^2\leq \|\Phi^*(Q_2^*)^{\top}-U^*\|_F^2.$$
Therefore,
\beaa
\|\bar{\Phi}\bar{X}_c-\Phi^*(Q_2^*)^{\top}\|_F^2 & = & \|\bar{\Phi}\bar{X}_c-U^*+U^*-\Phi^*(Q_2^*)^{\top}\|_F^2 \\
 & \leq & 2\|\bar{\Phi}\bar{X}_c-U^*\|_F^2+2\|\Phi^*(Q_2^*)^{\top}-U^*\|_F^2 \\
 & \leq & 2\alpha(k)\|\Phi^{opt}X_c^{opt}-U^*\|_F^2+2\|\Phi^*(Q_2^*)^{\top}-U^*\|_F^2 \\
 & \leq & 2(1+\alpha(k))\|\Phi^*(Q_2^*)^{\top}-U^*\|_F^2.
\eeaa

Denote by $Z_i$ the $i$th row of a matrix $Z$. We define the set that contains the potential errors
$$E_a = \bigg\{i\in C_a^*| \|\left(\bar{\Phi}\bar{X}_c\right)_i-\left(\Phi^*(Q_2^*)^{\top}\right)_i\|^2\geq\half\bigg\},$$
$$E = \cup_{a=1}^kE_a, S_a = C_a^*\backslash E_a.$$
A straightforward bound for the cardinality of $E$ gives
$$|E|\leq2\|\bar{\Phi}\bar{X}_c-\Phi^*(Q_2^*)^{\top}\|_F^2.$$
Now we define a partition of the community index set $[k] = \{1,...,k\},$
\beaa
& & B_1 = \{a \mid S_a = \emptyset\},\\
& & B_2 = \{a \mid \forall i,j\in S_a, (\bar{\Phi}\bar{X}_c)_i = (\bar{\Phi}\bar{X}_c)_j \},\\
& & B_3 = [k]\backslash(B_1\cup B_2).
\eeaa
For $\forall i,j \in E^c = \cup_{a\in B_2\cup B_3}S_a,$ and $i\in C_a^*, j\in C_b^*, a\neq b$, 
\beaa
& & \|(\bar{\Phi}\bar{X}_c)_i - (\bar{\Phi}\bar{X}_c)_j\| \\ & \geq & 
\|(\Phi^*(Q_2^*)^{\top})_i-(\Phi^*(Q_2^*)^{\top})_j\|  -\|(\bar{\Phi}\bar{X}_c)_i - (\Phi^*(Q_2^*)^{\top})_i \| \\ & & -  \|(\bar{\Phi}\bar{X}_c)_j - (\Phi^*(Q_2^*)^{\top})_j\|  \nonumber\\
 & > & \sqrt{2}-\frac{\sqrt{2}}{2}-\frac{\sqrt{2}}{2} =  0.
\eeaa

That is, $(\bar{\Phi}\bar{X}_c)_i \neq (\bar{\Phi}\bar{X}_c)_i$, or equivalently $\bar{\Phi}_i\neq \bar{\Phi}_j$. This implies that for all $S_a, a\in B_2$, they belong to different clusters. It further indicates that all nodes in $\cup_{a\in B_2}S_a$ are successfully classified with a proper permutation $\Pi^*$. In other words, 
$$Err_{\Pi^*}(\bar{\Phi})\subset \cup_{a\in B_3}S_a\cup E.$$ 

The above argument also reveals that if $a\neq b, S_a,S_b\neq \emptyset$, then $S_a$ and $S_b$ will correspond to different rows from $\bar{X}_c$. If $a\in B_2$, then $S_a$ corresponds to only one row of $\bar{X}_c$. If $a\in B_3$, then $S_a$ corresponds to at least two different rows of $\bar{X}_c$. But note that $\bar{X}_c$ contains at most $k$ different rows. Therefore,
$$|B_2|+2|B_3|\leq k = |B_1|+|B_2|+|B_3|,$$
which further implies 
$$|B_3|\leq|B_1|.$$
Note that $\cup_{a\in B_1}C_a^*\subset E$, we have $|B_1|n_{\min}\leq |E|$, or equivalently
$$|B_1|\leq \frac{1}{n_{\min}}|E|.$$
Combining all previous results yields
\beaa
|Err_{\Pi^*}(\bar{\Phi})| & \leq & |E| + |\cup_{a\in B_3}S_a| \leq  |E| + |B_3|n_{\max} \\
& \leq & (1+\frac{n_{\max}}{n_{\min}})|E| \\
& \leq & 4(1+\frac{n_{\max}}{n_{\min}})(1+\alpha(k))\|U^*Q_2^*-\Phi^*\|_F^2\\
& \leq & \frac{8\sqrt{k}n_{\max}}{\theta_{\min}n_{\min}}\left(1+\frac{n_{\max}}{n_{\min}}\right)(1+\alpha(k))\|\Delta\|_{1,\theta}^{\half}.
\eeaa
This completes the proof.   $\blacksquare$

Note that the factor $\frac{1}{\theta_{\min}}$ appears in the bound. This term
vanishes under the SBM where all $\theta_i=1$. However, it is undesirable under the
DCSBM. To avoid the dependence on this term, we propose to solve the weighted K-means problem (\ref{prob:w-k-means}). In this case, suppose that the solution to problem (\ref{prob:w-k-means}) is $\bar{\Phi}$, then the fair error measure will be  
$$\min_{\Pi}\sum_{i \in Err_{\Pi}(\bar{\Phi})}\theta_i^2,$$
where $\Pi$ is a permutation matrix. This measure means that when a node has smaller $\theta_i$, it potentially has few edges and provides less information for clustering, then misclassifying this node is somewhat forgivable and suffers less penalty. In contrast, if a node has large $\theta_i$, then misclassifying this node should incur more penalty. Now we present the theorem that address the above issue.

\begin{theorem}\label{theorem:dcsbm-w-kmeans}
	Suppose that $\bar{\Phi}$ is generated by some weighted K-means clustering
    algorithm with the number of clusters  $k$ equal to the number of
    communities. Then under the DCSBM,
	$$\min_{\Pi}\sum_{i \in Err_{\Pi}(\bar{\Phi})}\theta_i^2 \leq \frac{16\sqrt{k}J_{\max}}{H^2_{\min}J_{\min}}\left(C_H+ H^2_{\max}\right) (1+\alpha(k))\|\Delta\|_{1,\theta}^{\half},$$
	where $\alpha(k)$ is the approximation ratio of the polynomial time approximation algorithm for problem (\ref{prob:w-k-means}) and the quantity $C_H\ll H^2_{\max}$. 
\end{theorem}

The proof of this theorem is similar in nature to that of Theorem \ref{theorem:dcsbm-kmeans}, but is more technically involved. For the sake of being succinct, we omit the proof here. 

For large scale networks, solving problem (\ref{prob:CP_DCk}) can be very
efficient when the $\ell_0$ penalty parameter $p$ is set to be moderate and the
asynchronous parallel scheme is applied. However, performing the K-means or
weighted K-means clustering will be a bottleneck when the number of communities is relatively large. Therefore, the proposed direct rounding scheme will be desirable. In the following theorem, we discuss the theoretical guarantees for this rounding scheme. 
To invoke the analysis, the following assumption is needed. 
\begin{assumption}\label{assumption:round}
	Define the set $Err(Q,\delta) = \{i:\|(U^*Q)_i-\phi^*_i\|^2\geq \delta\}$ and define $T_a(Q,\delta) = C_a^*\backslash Err(Q,\delta)$. 
	Suppose $Q_1^*, Q_2^*$ are defined according to Theorem \ref{theorem:dcsbm-Frobenius-err}, we assume that all $T_a(Q_1^*,\frac{1}{32p^2})\neq \emptyset, T_a(Q_2^*,\frac{1}{32p^2}) \neq \emptyset.$
\end{assumption} 

This assumption states that for each community $C_a^*$, there is at least one node $i\in C_a^*$ such that $U_i^*$ is close to the true assignment $\phi_i^*$ after the rotation $Q_1^*$ or $Q_2^*$. In our extensive numerical experiments, this assumption actually always holds.

\begin{theorem}\label{theorem:dcsbm-rounding}
	Suppose that all Assumption \ref{assumption:round} holds and the community assignment $\bar{\Phi}$ is given by a directly rounding $U^*$. Then under the DCSBM, the error is bounded by
	
	$$\min_{\Pi}|Err_{\Pi}(\bar{\Phi})| \leq C_2\frac{p^{2}\sqrt{k}n_{\max}}{n_{\min}\theta_{\min}}\|\Delta\|_{1,\theta}^{\half}, $$
	$$\min_{\Pi}\sum_{i \in Err_{\Pi}(\bar{\Phi})}\theta_i^2 \leq C_2\frac{p^{2}\sqrt{k}M_{\max}}{J_{\min}}\|\Delta\|_{1,\theta}^{\half},$$
	where $C_2>0$ is some absolute constant.
\end{theorem}

\emph{Proof.} 
(i) We prove the first inequality by showing that all nodes in $\cup_{a=1}^kT_a(Q_2^*,\frac{1}{32p^2})$ are correctly classified under a proper permutation matrix $\Pi$.  Define $\delta_0 = \frac{1}{32p^2}$. By the Assumption \ref{assumption:round}, we have 
$$|Err(Q_2^*,\delta_0)| \leq \frac{1}{\delta_0}\|U^*Q_2^*-\Phi^*\|_F^2, \text{ and }T_a(Q_2^*,\delta_0)\neq \emptyset, 1\leq a \leq k.$$

First, for $\forall i,j\in T_a(Q_2^*,\delta_0)$, and for some $a\in[k]$, we have $\phi^*_i=\phi_j^*$. Recall that $(u^*_i)^{\top}$ stands for the $i$th row of the solution $U^*$ to problem (\ref{prob:CP_DCk}), we further have  
\beaa
\|u_i^*-u_j^*\|  & = & \|(Q_2^*)^{\top}(u_i^*-u_j^*)\| \\ 
& = & \|(Q_2^*)^{\top}u_i^*-\phi^*_i+\phi_j^*-(Q_2^*)^{\top}u_j^*\| \\
& \leq & \|(Q_2^*)^{\top}u_i^*-\phi^*_i\| + \|(Q_2^*)^{\top}u_j^* -\phi_j^*\|\\
& < & 2\sqrt{\delta_0}.
\eeaa 
Second, for $\forall i, j, a\neq b, i\in T_a(Q_2^*,\delta_0), j\in T_b(Q_2^*,\delta_0)$, we have $\phi^*_i\neq\phi_j^*$. Similarly, 
\beaa
\|u_i^*-u_j^*\|  & = & \|(Q_2^*)^{\top}(u_i^*-u_j^*)\| \\ 
& = & \|(Q_2^*)^{\top}u_i^*-\phi^*_i+\phi^*_i-\phi^*_j+\phi_j^*-(Q_2^*)^{\top}u_j^*\| \\
& \geq & \|\phi_i^*-\phi_j^*\|-\|(Q_2^*)^{\top}u_i^*-\phi^*_i\| - \|(Q_2^*)^{\top}u_j^*-\phi_j^*\|\\
& > & \sqrt{2}-2\sqrt{\delta_0}.
\eeaa
That is, 
\be\label{thm:round:sphere}
\|u_i^*-u_j^*\| \begin{cases}
	  < 2\sqrt{\delta_0}, & \forall a, \forall i,j \in T_a(Q_2^*,\delta_0), \\
	  >\sqrt{2}-2\sqrt{\delta_0}, & \forall a\neq b, \forall i\in T_a(Q_2^*,\delta_0), \forall j\in T_b(Q_2^*,\delta_0).
\end{cases}
\ee

From each set $T_a(Q_2^*,\delta_0)$, we take an arbitrary representative $v_a$. Then we have $k$ vectors in $\mathbb{R}^k$ such that 
$$\|v_i\|^2 = 1, \|v_i\|_0\leq p, v_i\geq 0, \forall 1\leq i\leq k, $$
$$\|v_i-v_j\|\geq \sqrt{2} - 2\sqrt{\delta_0}, \forall i\neq j,$$
which further implies that 
$$\langle v_i,v_j\rangle\leq \epsilon, \forall i\neq j,$$
where $\epsilon = \frac{1}{2p}\geq 2(\sqrt{2\delta_0}-\delta_0)$. The following proof consists of mainly three steps. 


\textbf{Step (a)}. Define $m(i) = \arg\max_{1\leq j\leq k}(v_i)_j$, where $(v_i)_j$ is the $j$th component of $v_i$. Then for any $i\neq j$, we prove $m(i)\neq m(j)$.  Suppose there exist $i\neq j$ such that $m(i) = m(j).$ Then by the definition of $m(i)$, it is straightforward that $$(v_i)_{m(i)}\geq\frac{1}{\sqrt{p}}.$$
Therefore, by $v_i,v_j\geq 0$, 
 $$\langle v_i,v_j\rangle \geq (v_i)_{m(i)}\cdot (v_j)_{m(j)}\geq \frac{1}{p}>\frac{1}{2p} = \epsilon,$$
 which leads to a contradiction. We can then choose a proper permutation of
 index $\Pi$ such that $m(i) = i$ for $1\leq i\leq k$. This means that all the
 $k$ representatives $v_1,...,v_j$ are correctly classified.

\textbf{Step (b)}. Suppose that after a proper permutation of indices, $m(i) = i, 1\leq i\leq k$. 
Then in this step, we prove that for all $1\leq i\leq k,$ $v_i$ is sufficiently close to $e_i$. First, by
\be\label{thm:round_1}
(v_1)_2\cdot \frac{1}{\sqrt{p}}\leq (v_1)_2\cdot (v_2)_2\leq\langle v_1,v_2\rangle\leq \frac{1}{2p},
\ee
we obtain, 
\bea
(v_1)_2\leq \frac{1}{2\sqrt{p}}.
\eea
Repeat the same argument for $(v_1)_3,...,(v_1)_k$ and we have them all less
than or equal to $\frac{1}{2\sqrt{p}}$. It follows from $\|v_1\|_0\leq p$ and
$\|v_1\|^2 = 1$ that 
\be
\label{thm:round_3}
(v_1)_1^2 = 1-\sum_{j=2}^k(v_1)_j^2\geq 1-\frac{p-1}{4p}>\frac{3}{4}.
\ee
By repeating the same argument for $v_2,...,v_k$, we have for $1\leq i\leq k$,
\be
\label{thm:round_4}
\begin{cases}
(v_i)_i>\frac{\sqrt{3}}{2},\\
(v_i)_j<\frac{1}{2\sqrt{p}},  & \text{ for } \forall j\neq i.
\end{cases}
\ee
Now based on the new lower bounds of $(v_i)_i,1\leq i\leq k,$ we repeat the
process (\ref{thm:round_1}), (\ref{thm:round_3}) and (\ref{thm:round_4}) again. We obtain that for all $1\leq i\leq k$,
$$\begin{cases}
(v_i)_i>\sqrt{1-\frac{1}{3p}},\\
(v_i)_j<\frac{1}{\sqrt{3}p}, \text{ for } \forall j\neq i.
\end{cases}$$

\textbf{Step (c)}. Now we argue that all nodes in $\cup_{1\leq a\leq k}T_a(Q_2^*,\delta_0)$ are correctly classified in this step.
For any other nodes $s\in T_a(Q_2^*,\delta_0)$, due to its proximity to $v_a$ (\ref{thm:round:sphere}), we know that 
$$\begin{cases}
(u^*_s)_a>\sqrt{1-\frac{1}{3p}}-\frac{\sqrt{2}}{4p},\\
(u^*_s)_j<\frac{1}{\sqrt{3}p}+\frac{\sqrt{2}}{4p}, \forall j\neq a.
\end{cases}$$
For all $p\geq 2$, $u^*_s(a)>u^*_s(j)$ for $j\neq a$. This means that all nodes in $\cup_{1\leq a\leq k}T_a(Q_2^*,\delta_0)$ are correctly classified. That is, 
\beaa
\min_{\Pi}|Err_{\Pi}(\bar{\Phi})| & \leq & |Err(Q_2^*,\delta_0)|\leq \frac{1}{\delta_0}\cdot\|U^*Q_2^*-\Phi^*\|_F^2 \\
& \leq & C_2\frac{p^{2}\sqrt{k}n_{\max}}{n_{\min}\theta_{\min}}\|\Delta\|_{1,\theta}^{\half},
\eeaa
where the last inequality follows from the result of Theorem \ref{theorem:dcsbm-Frobenius-err}.

(ii). The second inequality of this theorem can be extended directly from the first one: 
\beaa
\sum_{i \in Err(Q_1^*,\delta_0)}\theta_i^2 \leq \sum_{i \in Err(Q_1^*,\delta_0)}\theta_i^2\cdot\frac{1}{\delta_0}\cdot\|(Q_1^*)^{\top}u_i^*-\phi_i^*\|^2 
 \leq \frac{1}{\delta_0}\|\Theta(U^*Q_1^*-\Phi^*)\|_F^2.
\eeaa
Then by following the bound in Theorem \ref{theorem:dcsbm-Frobenius-err}, the
proof is completed. 
$\blacksquare$

\section{Numerical Result}\label{section:numericalEXP}
In this section, we evaluate the performance of our RBR
method on both synthetic and real datasets, where
 the parameter $\lambda$ is set to $1/\|d\|_1$ in
all cases.
 Our RBR algorithm with rounding and K-means are referred to  as RBR(r) and RBR(K), 
 respectively.  For
synthetic and small-sized real networks, we compare them with the CMM algorithm in \cite{CMM-X.Li-2015}, 
the SCORE algorithm in \cite{Spectral-SCORE-Jin-2015} and the OCCAM algorithm in \cite{Spectral-OCCAM-Zhang-2014}. For large-sized real datasets, 
 we only compare the asynchronous parallel RBR algorithm with the LOUVAIN algorithm in
 \cite{2008JSMTE..10..008B}, which is an extremely fast algorithm to process large networks.

\subsection{Solvers and Platform}
The algorithm RBR(r) is implemented in C, with  multi-threading
support using OpenMP. Due to the usage of K-means, the version
RBR(K) is written in MATLAB. The CMM, SCORE and OCCAM for synthetic and small-sized
real networks are also written in  MATLAB.  
For large-sized
networks, we use the LOUVAIN solver hosted
on Google Sites\footnote{See \url{https://sites.google.com/site/findcommunities/}},
which is implemented in C++.
 The parameters in all methods are set to their default values unless otherwise specified. Since $k$, the number of communities of the synthetic and small-sized real networks, is known (usually very small), the parameter $r$, i.e., the number of columns of $U$ is set to $k$. The parameter $p$ in the $\ell_0$ constraints of our RBR algorithm is also set to $k$ in these tests.  
 A few different values of $p$ are tested on large-size real networks. 

All numerical experiments are performed on a workstation with two twelve-core Intel Xeon E5-2680 v3 processors at 2.5 GHz
 and a total amount of 128 GB shared memory. The efficiency of OpenMP
may be affected by NUMA systems greatly. It is much slower for a CPU to 
access the memory of another CPU. To prevent such allocations of
threads, 
we bind all programs to one CPU, i.e., up to 12 physical cores are available
per task.

Other programming environments are:
\begin{itemize}
	\item gcc/g++ version 6.2.0 is used for compiling our C programs and LOUVAIN executables.
	\item MATLAB version R2015b is used for running CMM, SCORE, OCCAM and RBR with K-means solvers.
	\item Intel MKL 2017 for BLAS routines, linked with {\ttfamily libmkl\_sequential.so}.
\end{itemize}
We should mention that there are threaded and sequential versions in the Intel MKL's library.
 The threaded MKL library should be disabled in favor of our manual parallelization.
The reason is that when updating one row, the threaded library will automatically utilize all
available computing threads for BLAS, which may 
interfere with the cores processing other rows.
Besides, the sequential library is usually faster than
the parallel library with one thread.
 The reported runtimes are wall-clock times in seconds.

\subsection{Results on Synthetic Data}
To set up the experiments, we generate the synthetic graph of $n=m\times k$ nodes and split them into $k$ groups $C_1^*,\ldots, C_k^*$, each group with $m$ nodes. These groups are then used as the ground truth of the model. The edges of the graph are generated by sampling. For each pair of nodes $i\in C_a^*$ and $j\in C_b^*$, $i$
and $j$ are connected
with probability $\min\{1,\theta_i\theta_jB_{ab}\}$, where
\begin{equation}
B_{ab}= \begin{cases} 
q, & a = b, \\
0.3q, & a \neq b.
\end{cases}   
\end{equation}
For each $i$, $\theta_i$ is sampled independently from a $\mathrm{Pareto}(\alpha, \beta)$ distribution
with probability density function $f(x;\alpha,\beta)=\frac{\alpha\beta^\alpha}{x^{\alpha+1}\mathbf{1}_{\{x\geqslant \beta\}}}$.
Here $\alpha$ and $\beta$ are called the \emph{shape} and \emph{scale} parameters respectively.
In our testing model, we choose different shape
parameter $\alpha$ first, and select the scale parameter
$\beta$ such that $\mathbb{E}(\theta_i)=1$ for all $i$.
These parameters determine the strength of the cluster structure of the graph.
With larger $q$ and $\alpha$, it is easier to detect the community structure. Otherwise
all these algorithms may fail. Since the ground truths of the synthetic networks are known, we compute the misclassification
rate to check the correctness of these algorithms. Suppose that $C_1,\ldots,C_k$ are the detected communities provided
by any algorithm, then the misclassification rate can be defined as
\begin{equation} \label{eqn:mis}
\mathrm{err}:=1-\frac{\sum_{i}^{k}\max_j|C_i\cap C_{j}^*|}{n}.
\end{equation}
We shall mention that permuting
the columns of $U$ does not change its objective function value. 
Since it is impractical to  enumerate all possible permutations and choose the best one to match the
detected communities and the ground truth when
$k$ is large, we still use the definition of misclassification rate in \eqref{eqn:mis}, and it works in most cases during our numerical experiments. 

\begin{figure}[htb]
	\caption{Numerical results on synthetic data.} \label{fig:syn}
	\centering
	\hfill
	\subfigure[2 communities, 200 nodes per community]{
		\begin{tikzpicture}[scale=0.7]
		
\begin{axis}[
grid,
xmin=1.1,
xlabel=shape,
ylabel=misclassification rate,
legend style={nodes={inner sep=0.5pt}}
]

\addplot[color=red,mark=x] coordinates{(1.1, 0.41975) (1.2, 0.35713) (1.3, 0.22625) (1.4, 0.16187) (1.5, 0.14338) (1.6, 0.10112) (1.7000000000000002, 0.09025) (1.8, 0.080625) (1.9, 0.09025)};
\addlegendentry[font=\tiny,anchor=west]{{\ttfamily q=0.05,CMM}}

\addplot[color=green,mark=x] coordinates{(1.1, 0.3815) (1.2, 0.19587) (1.3, 0.078375) (1.4, 0.046375) (1.5, 0.0325) (1.6, 0.022125) (1.7000000000000002, 0.01575) (1.8, 0.01675) (1.9, 0.0115)};
\addlegendentry[font=\tiny,anchor=west]{{\ttfamily q=0.1,CMM}}

\addplot[color=blue,mark=x] coordinates{(1.1, 0.3025) (1.2, 0.092) (1.3, 0.034) (1.4, 0.018375) (1.5, 0.010625) (1.6, 0.009) (1.7000000000000002, 0.004375) (1.8, 0.003375) (1.9, 0.0025)};
\addlegendentry[font=\tiny,anchor=west]{{\ttfamily q=0.15,CMM}}

\addplot[color=black,mark=x] coordinates{(1.1, 0.1915) (1.2, 0.06425) (1.3, 0.02025) (1.4, 0.0085) (1.5, 0.00375) (1.6, 0.002625) (1.7000000000000002, 0.001375) (1.8, 0.001375) (1.9, 0.00125)};
\addlegendentry[font=\tiny,anchor=west]{{\ttfamily q=0.2,CMM}}

\addplot[color=red,mark=o] coordinates{(1.1, 0.39449999999999996) (1.2, 0.25649999999999995) (1.3, 0.153) (1.4, 0.1025) (1.5, 0.0875) (1.6, 0.0605) (1.7000000000000002, 0.0685) (1.8, 0.07050000000000001) (1.9, 0.058499999999999996)};
\addlegendentry[font=\tiny,anchor=west]{{\ttfamily q=0.05,RBR(r)}}

\addplot[color=green,mark=o] coordinates{(1.1, 0.2845) (1.2, 0.11000000000000001) (1.3, 0.052500000000000005) (1.4, 0.027500000000000004) (1.5, 0.0315) (1.6, 0.018000000000000002) (1.7000000000000002, 0.014000000000000002) (1.8, 0.01) (1.9, 0.009)};
\addlegendentry[font=\tiny,anchor=west]{{\ttfamily q=0.1,RBR(r)}}

\addplot[color=blue,mark=o] coordinates{(1.1, 0.19849999999999998) (1.2, 0.0485) (1.3, 0.025500000000000002) (1.4, 0.015000000000000003) (1.5, 0.009) (1.6, 0.003999999999999999) (1.7000000000000002, 0.0035000000000000005) (1.8, 0.0055) (1.9, 0.0025)};
\addlegendentry[font=\tiny,anchor=west]{{\ttfamily q=0.15,RBR(r)}}

\addplot[color=black,mark=o] coordinates{(1.1, 0.11599999999999999) (1.2, 0.0395) (1.3, 0.018500000000000003) (1.4, 0.008000000000000002) (1.5, 0.002) (1.6, 0.0035000000000000005) (1.7000000000000002, 0.003) (1.8, 0.0) (1.9, 0.0015)};
\addlegendentry[font=\tiny,anchor=west]{{\ttfamily q=0.2,RBR(r)}}

\end{axis}

		\end{tikzpicture}
	}
	\hfill
	\subfigure[2 communities, 450 nodes per community]{
		\begin{tikzpicture}[scale=0.7]
		
\begin{axis}[
grid,
xmin=1.1,
xlabel=shape,
ylabel=misclassification rate,
legend style={nodes={inner sep=0.5pt}}
]

\addplot[color=red,mark=x] coordinates{(1.1, 0.35978) (1.2, 0.11889) (1.3, 0.074889) (1.4, 0.044444) (1.5, 0.023111) (1.6, 0.013778) (1.7000000000000002, 0.012) (1.8, 0.0084444) (1.9, 0.0073333)};
\addlegendentry[font=\tiny,anchor=west]{{\ttfamily q=0.05,CMM}}

\addplot[color=green,mark=x] coordinates{(1.1, 0.22) (1.2, 0.026) (1.3, 0.014) (1.4, 0.004) (1.5, 0.0011111) (1.6, 0.00088889) (1.7000000000000002, 0.00066667) (1.8, 0.00022222) (1.9, 0.00044444)};
\addlegendentry[font=\tiny,anchor=west]{{\ttfamily q=0.1,CMM}}

\addplot[color=blue,mark=x] coordinates{(1.1, 0.11444) (1.2, 0.030444) (1.3, 0.0026667) (1.4, 0.00066667) (1.5, 0.00044444) (1.6, 0.00022222) (1.7000000000000002, 0.00022222) (1.8, 0.00022222) (1.9, 0)};
\addlegendentry[font=\tiny,anchor=west]{{\ttfamily q=0.15,CMM}}

\addplot[color=black,mark=x] coordinates{(1.1, 0.071778) (1.2, 0.0055556) (1.3, 0.0055556) (1.4, 0.00044444) (1.5, 0.00044444) (1.6, 0.00022222) (1.7000000000000002, 0.00088889) (1.8, 0.00066667) (1.9, 0.00022222)};
\addlegendentry[font=\tiny,anchor=west]{{\ttfamily q=0.2,CMM}}

\addplot[color=red,mark=o] coordinates{(1.1, 0.2355556) (1.2, 0.07977759999999999) (1.3, 0.052000000000000005) (1.4, 0.0304446) (1.5, 0.013333399999999999) (1.6, 0.012) (1.7000000000000002, 0.008666799999999999) (1.8, 0.0071112) (1.9, 0.0084446)};
\addlegendentry[font=\tiny,anchor=west]{{\ttfamily q=0.05,RBR(r)}}

\addplot[color=green,mark=o] coordinates{(1.1, 0.1275558) (1.2, 0.0215556) (1.3, 0.0104446) (1.4, 0.004222) (1.5, 0.001111) (1.6, 0.0004444) (1.7000000000000002, 0.0008888) (1.8, 0.0002222) (1.9, 0.0004444)};
\addlegendentry[font=\tiny,anchor=west]{{\ttfamily q=0.1,RBR(r)}}

\addplot[color=blue,mark=o] coordinates{(1.1, 0.083111) (1.2, 0.0184446) (1.3, 0.0024441999999999997) (1.4, 0.0008888) (1.5, 0.0004444) (1.6, 0.0002222) (1.7000000000000002, 0.0002222) (1.8, 0.0) (1.9, 0.0)};
\addlegendentry[font=\tiny,anchor=west]{{\ttfamily q=0.15,RBR(r)}}

\addplot[color=black,mark=o] coordinates{(1.1, 0.0444446) (1.2, 0.004222000000000001) (1.3, 0.0039998) (1.4, 0.0004444) (1.5, 0.0004444) (1.6, 0.0002222) (1.7000000000000002, 0.0008888) (1.8, 0.0006666) (1.9, 0.0002222)};
\addlegendentry[font=\tiny,anchor=west]{{\ttfamily q=0.2,RBR(r)}}

\end{axis}

		\end{tikzpicture}
	}
	
	\vspace{.75cm}
	
	\centering
	\hfill
	\subfigure[3 communities, 200 nodes per community]{
		\begin{tikzpicture}[scale=0.7]
		
\begin{axis}[
grid,
xmin=1.1,
xlabel=shape,
ylabel=misclassification rate,
legend style={nodes={inner sep=0.5pt}}
]

\addplot[color=red,mark=x] coordinates{(1.1, 0.384) (1.2, 0.142) (1.3, 0.046333) (1.4, 0.027333) (1.5, 0.015333) (1.6, 0.015667) (1.7000000000000002, 0.0093333) (1.8, 0.0053333) (1.9, 0.006)};
\addlegendentry[font=\tiny,anchor=west]{{\ttfamily q=0.15,CMM}}

\addplot[color=green,mark=x] coordinates{(1.1, 0.226) (1.2, 0.054) (1.3, 0.013333) (1.4, 0.009) (1.5, 0.0036667) (1.6, 0.001) (1.7000000000000002, 0.00066667) (1.8, 0.00066667) (1.9, 0.00066667)};
\addlegendentry[font=\tiny,anchor=west]{{\ttfamily q=0.25,CMM}}

\addplot[color=blue,mark=x] coordinates{(1.1, 0.178) (1.2, 0.020333) (1.3, 0.003) (1.4, 0.0026667) (1.5, 0.003) (1.6, 0.001) (1.7000000000000002, 0.0016667) (1.8, 0.0013333) (1.9, 0.001)};
\addlegendentry[font=\tiny,anchor=west]{{\ttfamily q=0.35,CMM}}

\addplot[color=black,mark=x] coordinates{(1.1, 0.11433) (1.2, 0.014333) (1.3, 0.0036667) (1.4, 0.0013333) (1.5, 0.0026667) (1.6, 0.002) (1.7000000000000002, 0.001) (1.8, 0.0023333) (1.9, 0.0016667)};
\addlegendentry[font=\tiny,anchor=west]{{\ttfamily q=0.45,CMM}}

\addplot[color=red,mark=o] coordinates{(1.1, 0.32633340000000005) (1.2, 0.1210002) (1.3, 0.048666600000000004) (1.4, 0.026666600000000002) (1.5, 0.016666800000000002) (1.6, 0.014333400000000001) (1.7000000000000002, 0.0106666) (1.8, 0.0053334) (1.9, 0.0053332)};
\addlegendentry[font=\tiny,anchor=west]{{\ttfamily q=0.15,RBR(r)}}

\addplot[color=green,mark=o] coordinates{(1.1, 0.1826666) (1.2, 0.05233340000000001) (1.3, 0.013666800000000002) (1.4, 0.0093334) (1.5, 0.004) (1.6, 0.0013334) (1.7000000000000002, 0.0006668) (1.8, 0.0006668) (1.9, 0.0006668)};
\addlegendentry[font=\tiny,anchor=west]{{\ttfamily q=0.25,RBR(r)}}

\addplot[color=blue,mark=o] coordinates{(1.1, 0.1419998) (1.2, 0.022) (1.3, 0.0033334000000000003) (1.4, 0.003) (1.5, 0.0026668000000000004) (1.6, 0.001) (1.7000000000000002, 0.0016666) (1.8, 0.0013332) (1.9, 0.0010002000000000001)};
\addlegendentry[font=\tiny,anchor=west]{{\ttfamily q=0.35,RBR(r)}}

\addplot[color=black,mark=o] coordinates{(1.1, 0.08933340000000001) (1.2, 0.0116666) (1.3, 0.0046666) (1.4, 0.0013334) (1.5, 0.0026666000000000003) (1.6, 0.0026668) (1.7000000000000002, 0.0003334) (1.8, 0.0023331999999999997) (1.9, 0.0016668000000000002)};
\addlegendentry[font=\tiny,anchor=west]{{\ttfamily q=0.45,RBR(r)}}

\end{axis}

		\end{tikzpicture}
	}
	\hfill
	\subfigure[4 communities, 200 nodes per community]{
		\begin{tikzpicture}[scale=0.7]
		
\begin{axis}[
grid,
xmin=1.1,
xlabel=shape,
ylabel=misclassification rate,
legend style={nodes={inner sep=0.5pt}}
]

\addplot[color=red,mark=x] coordinates{(1.1, 0.604) (1.2, 0.24675) (1.3, 0.13675) (1.4, 0.0685) (1.5, 0.0575) (1.6, 0.051) (1.7000000000000002, 0.04025) (1.8, 0.02375) (1.9, 0.024)};
\addlegendentry[font=\tiny,anchor=west]{{\ttfamily q=0.1,CMM}}

\addplot[color=green,mark=x] coordinates{(1.1, 0.2655) (1.2, 0.083) (1.3, 0.03375) (1.4, 0.0155) (1.5, 0.00875) (1.6, 0.007) (1.7000000000000002, 0.0025) (1.8, 0.00275) (1.9, 0.0015)};
\addlegendentry[font=\tiny,anchor=west]{{\ttfamily q=0.2,CMM}}

\addplot[color=blue,mark=x] coordinates{(1.1, 0.23375) (1.2, 0.0395) (1.3, 0.01475) (1.4, 0.00475) (1.5, 0.0025) (1.6, 0.00175) (1.7000000000000002, 0.00075) (1.8, 0.0005) (1.9, 0.001)};
\addlegendentry[font=\tiny,anchor=west]{{\ttfamily q=0.3,CMM}}

\addplot[color=black,mark=x] coordinates{(1.1, 0.1415) (1.2, 0.02375) (1.3, 0.00725) (1.4, 0.0035) (1.5, 0.00275) (1.6, 0.0015) (1.7000000000000002, 0.00125) (1.8, 0.0015) (1.9, 0.00175)};
\addlegendentry[font=\tiny,anchor=west]{{\ttfamily q=0.4,CMM}}

\addplot[color=red,mark=o] coordinates{(1.1, 0.5309999999999999) (1.2, 0.21575000000000003) (1.3, 0.11950000000000001) (1.4, 0.06899999999999999) (1.5, 0.06) (1.6, 0.04575) (1.7000000000000002, 0.04125) (1.8, 0.026500000000000003) (1.9, 0.025)};
\addlegendentry[font=\tiny,anchor=west]{{\ttfamily q=0.1,RBR(r)}}

\addplot[color=green,mark=o] coordinates{(1.1, 0.27075) (1.2, 0.0825) (1.3, 0.0305) (1.4, 0.016) (1.5, 0.007999999999999998) (1.6, 0.007250000000000001) (1.7000000000000002, 0.0025) (1.8, 0.0025) (1.9, 0.0015)};
\addlegendentry[font=\tiny,anchor=west]{{\ttfamily q=0.2,RBR(r)}}

\addplot[color=blue,mark=o] coordinates{(1.1, 0.21625) (1.2, 0.03525) (1.3, 0.013250000000000001) (1.4, 0.004750000000000001) (1.5, 0.0022500000000000003) (1.6, 0.0015) (1.7000000000000002, 0.00075) (1.8, 0.0005) (1.9, 0.00075)};
\addlegendentry[font=\tiny,anchor=west]{{\ttfamily q=0.3,RBR(r)}}

\addplot[color=black,mark=o] coordinates{(1.1, 0.12175) (1.2, 0.02425) (1.3, 0.00775) (1.4, 0.0032500000000000003) (1.5, 0.00275) (1.6, 0.002) (1.7000000000000002, 0.00125) (1.8, 0.0015) (1.9, 0.0017500000000000003)};
\addlegendentry[font=\tiny,anchor=west]{{\ttfamily q=0.4,RBR(r)}}

\end{axis}

		\end{tikzpicture}
	}
	
\end{figure}

The performance of our method is shown in Figure \ref{fig:syn}. The y-axis is
the misclassification rate while the x-axis is the shape parameter $\alpha$ by
which we generate the random graphs. The circle solid line is our RBR algorithm with rounding. The line with ``x'' marker stands for the
CMM algorithm. 
The $q$ in the figure denotes the intra-cluster connection probability. The inter-cluster connection probability is set to be $0.3q$.  It is clear that in all these cases, our approach outperforms the CMM algorithm, especially when in class connection probability is small, i.e., the class structure is not very strong.
Table \ref{tab:syn01} and \ref{tab:syn02} are the results under $q=0.1,\alpha=1.4$ and
$q=0.1,\alpha=1.8$, respectively. We can see that CMM and our method have similar
performances. Both of them are better than SCORE and OCCAM. If the shape parameter
is increased from 1.4 to 1.8, all of the four methods perform better, but our
method still has the smallest misclassification rate.
\begin{table}[!htpb]
	\centering
	\caption{Misclassification rate and runtime, $q=0.1,\alpha=1.4$} \label{tab:syn01}
	\begin{tabular}{r||cc||cc||cc||cc}
		\hline
        
 & \multicolumn{2}{c||}{$n=200,k=2$} & \multicolumn{2}{c||}{$n=450,k=2$} & \multicolumn{2}{c||}{$n=200,k=3$} & \multicolumn{2}{c}{$n=200,k=4$}\\ \cline{2-9}
solver & err & Time & err & Time & err & Time & err & Time\\ \hline
CMM &1.06\% & 5.10 & 0.04\% & 22.05 & 0.30\% & 10.43 & 0.25\% & 18.37\\
RBR(r) &0.90\% & 0.01 & 0.04\% & 0.03 & 0.27\% & 0.03 & 0.23\% & 0.07\\
RBR(K) &1.05\% & 1.35 & 0.04\% & 4.33 & 0.30\% & 3.00 & 0.25\% & 4.37\\
OCCAM &1.50\% & 0.02 & 20.02\% & 0.05 & 19.13\% & 0.04 & 22.35\% & 0.08\\
SCORE &1.50\% & 0.02 & 18.02\% & 0.03 & 7.43\% & 0.05 & 14.90\% & 0.09\\
\hline  

	\end{tabular}
\end{table}

\begin{table}[!htpb]
	\centering
	\caption{Misclassification rate and runtime, $q=0.1,\alpha=1.8$} \label{tab:syn02}
	\begin{tabular}{r||cc||cc||cc||cc}
		\hline

 & \multicolumn{2}{c||}{$n=200,k=2$} & \multicolumn{2}{c||}{$n=450,k=2$} & \multicolumn{2}{c||}{$n=200,k=3$} & \multicolumn{2}{c}{$n=200,k=4$}\\ \cline{2-9}
solver & err & Time & err & Time & err & Time & err & Time\\ \hline
CMM &0.25\% & 5.20 & 0.00\% & 22.77 & 0.10\% & 10.44 & 0.10\% & 18.66\\
RBR(r) &0.25\% & 0.01 & 0.00\% & 0.04 & 0.10\% & 0.03 & 0.07\% & 0.06\\
RBR(K) &0.15\% & 1.41 & 0.00\% & 4.76 & 0.10\% & 3.24 & 0.10\% & 5.08\\
OCCAM &0.45\% & 0.02 & 0.00\% & 0.04 & 0.13\% & 0.04 & 2.88\% & 0.06\\
SCORE &0.45\% & 0.02 & 0.02\% & 0.03 & 0.27\% & 0.04 & 2.12\% & 0.05\\
\hline 
	\end{tabular}
\end{table}

\subsection{Results on Small-Scaled Real World Data}

In this section, we test the empirical performances of CMM,
SCORE, OCCAM, RBR(r) and RBR(K) on the US political
blog network dataset from \cite{snapnets}, the Simmons College network 
and the Caltech network from Facebook dataset. The properties of these
networks are shown in Table \ref{tab:small-networks}, where $nc$ is
the number of real communities.

\begin{table}[h] 
\centering
\caption{Small-scale real-world networks with ground truth}
\label{tab:small-networks}
\begin{tabular}{|l||r|r|r|l|}
\hline
Name      & nodes & edges  & nc & feature of community stucture         \\ \hline
Polblogs  & 1222 & 16714 & 2  & political leaning                     \\ 
Simmons & 1168 & 24449 & 4  & graduation year between 2006 and 2009 \\ 
Caltech & 597  & 12823 & 8  & total 8 dorm numbers                  \\  \hline
\end{tabular}
\end{table}

We run our RBR algorithms for $10$ times as a batch starting from
 different randomly initial points. The one with the highest modularity
is chosen as the result in this batch. This procedure is repeated $300$ times.  The histograms of the misclassification rates of RBR(r) are shown in Figure
\ref{hist-all}. 
 A summary of the
 misclassification rates and the cpu time of all algorithms is presented in Table
 \ref{comp_s}, where the results for RBR(r) and RBR(K) are the averaged values.
 \begin{table}[!htb]
\setlength{\abovecaptionskip}{0.cm}
\setlength{\belowcaptionskip}{-0.cm}
\centering
\caption{Overall comparison between the five methods}
\label{comp_s}
\begin{tabular}{|l||c|c||c|c||c|c|}
\hline
\multirow{2}{*}{} & \multicolumn{2}{c||}{Polblogs} & \multicolumn{2}{c||}{Simmons} & \multicolumn{2}{c|}{Caltech} \\ \cline{2-7} 
                  & err    & Time    & err    & Time   & err    & Time   \\ \hline
CMM               & 4.99\%        & 50.12        & 13.10\%       & 53.74       & 21.94\%       & 14.22        \\ 
SCORE             & 4.75\%        & 0.05          & 23.92\%       & 0.11         & 27.18\%       & 0.17         \\ 
OCCAM             & 5.32\%        & 0.10          & 24.36\%       & 0.16         & 35.53\%       & 0.10         \\ 
RBR(K)     & 5.09\%        & 1.42          & 12.56\%       & 10.53        & 16.51\%       & 1.08         \\
RBR(r)      & 4.75\%        & 0.04          & 14.00\%       & 0.08         & 21.03\%       & 0.06         \\ 
\hline
\end{tabular}
\end{table}

 On the polblogs network, the smallest misclassification rate  of  RBR(r) method is 4.4\% while its averaged value is 4.75\%.
 For the  Caltech network, we use dorms as the ground truth cluster
labels according to \cite{Traud}. The nodes are partitioned into 8 groups. The
averaged misclassfication rate of RBR(K) and RBR(r) are  16.54\% and 21.03\%,
respectively, whereas  CMM, SCORE and OCCAM have higher error rate of 21.94\%,
27.18\% and 35.53\%, respectively. In particular, the smallest error by RBR(K)
is 15.55\%.
  For the Simmons College network, the graduation year ranging from 2006 to 2009 is
  treated as cluster labels for its strong correlation with the community
  structure observed in \cite{Traud}. RBR(r), RBR(K), CMM, SCORE and OCCAM
  misclassified 14.00\%, 12.56\%, 13.10\%, 23.92\% and 24.36\% of the nodes on
  average, respectively. The smallest error of RBR(K) reaches 11.5\%.
\begin{figure}[!htb]
	\caption{Histograms of misclassification rates}
	\label{hist-all}
		\includegraphics[width=0.45\textwidth]{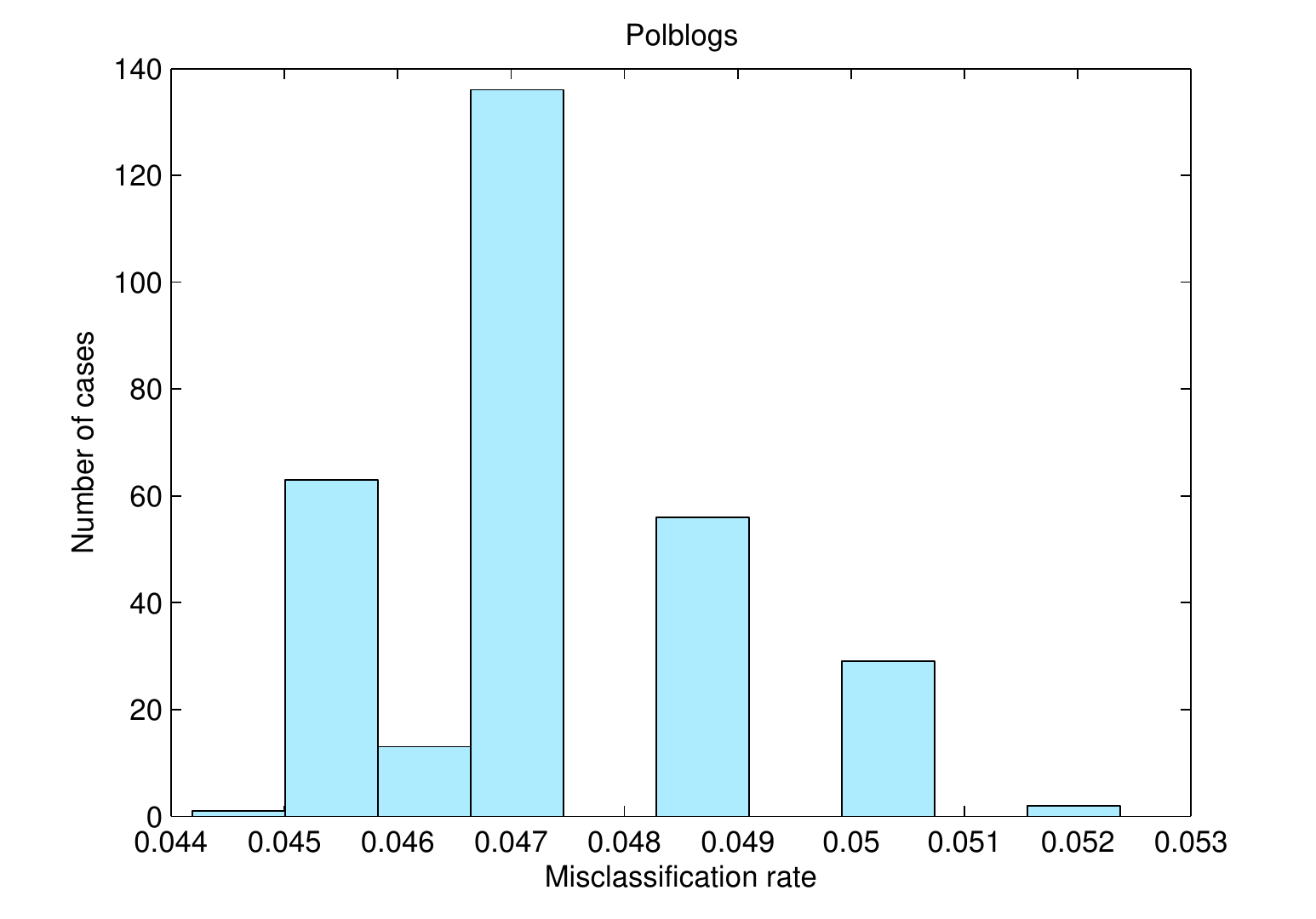}
		\includegraphics[width=0.45\textwidth]{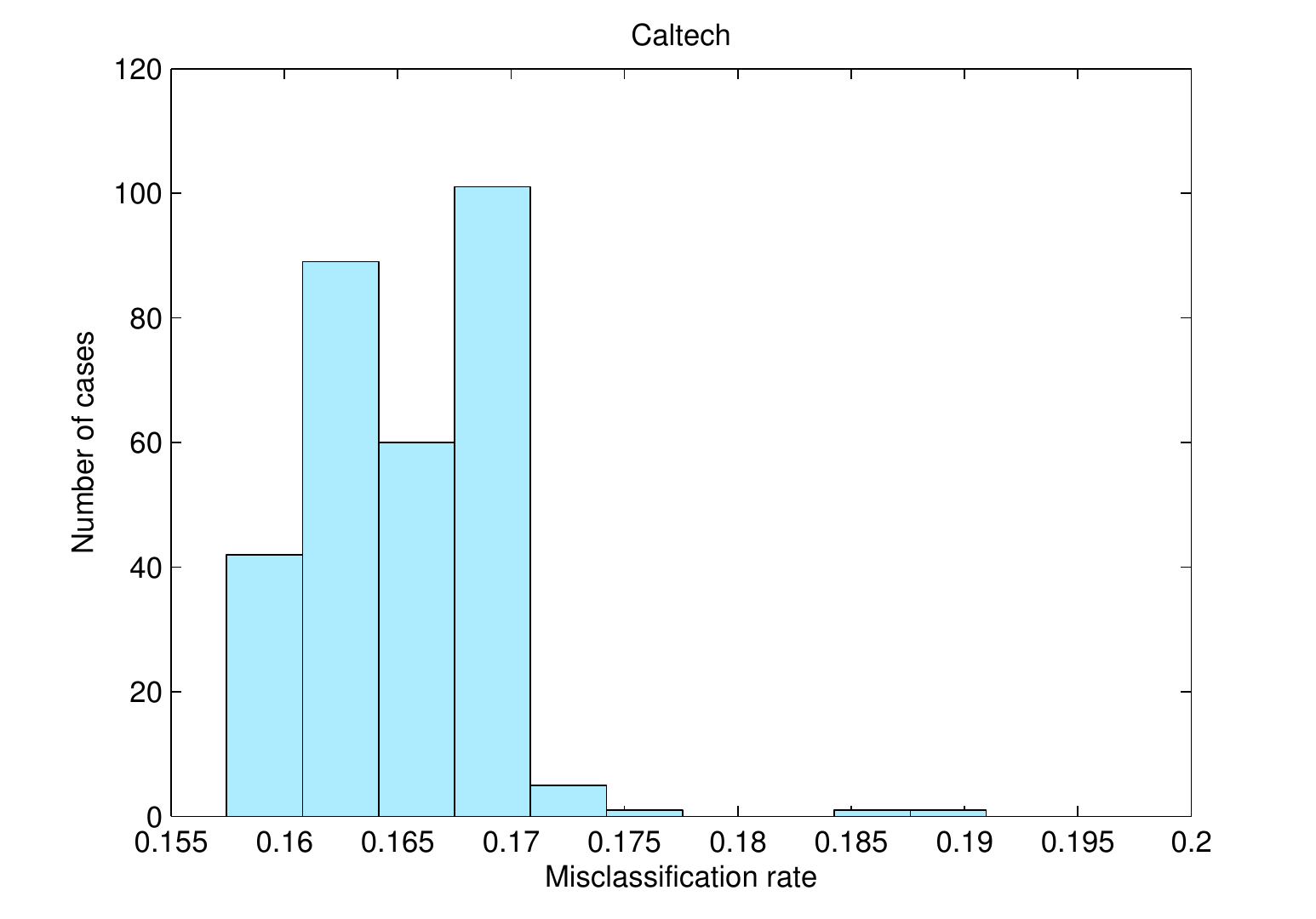}\\
        \begin{center} \includegraphics[width=0.45\textwidth]{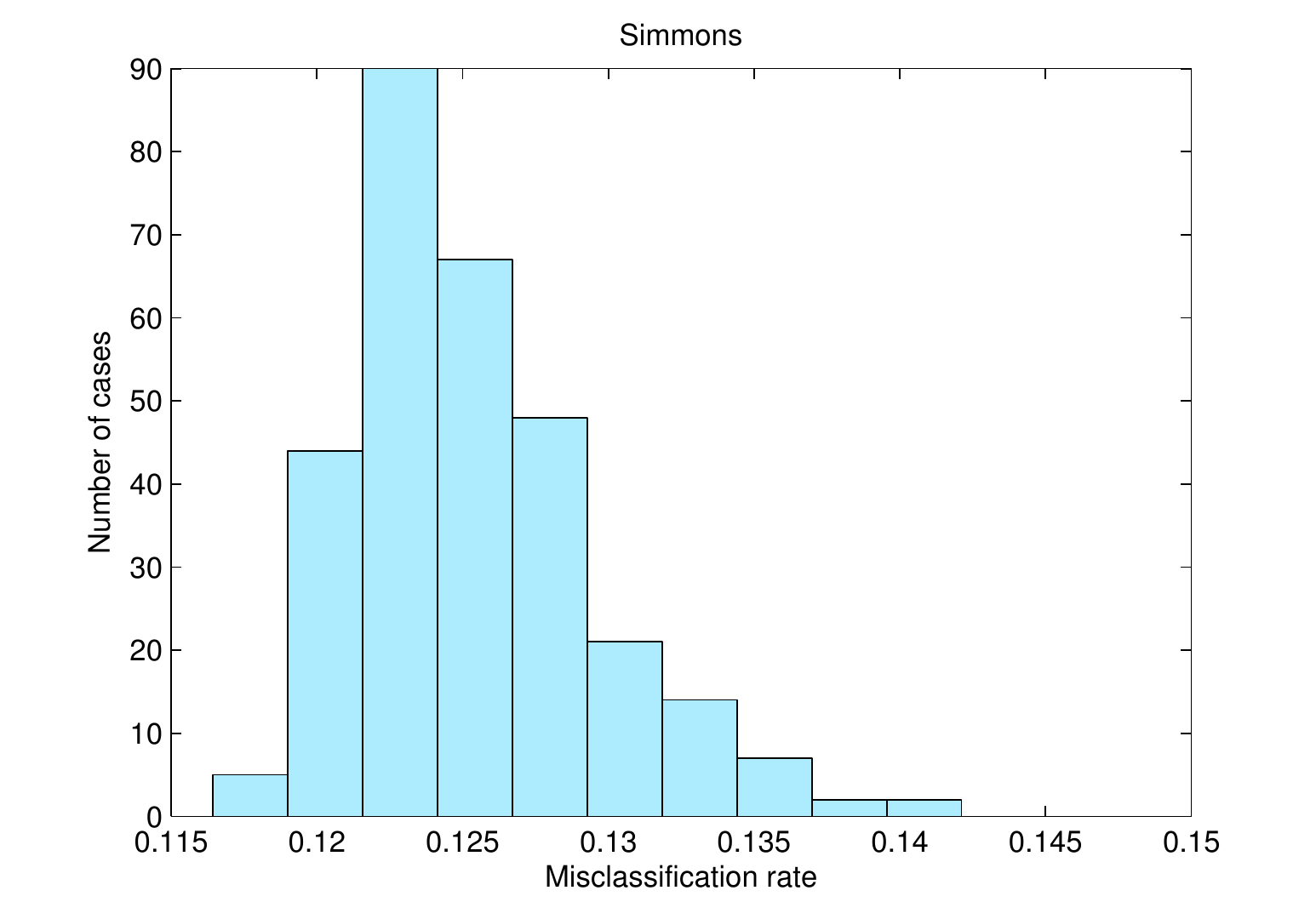}
        \end{center}
\end{figure}



Finally, the confusion matrices of CMM, SCORE and RBR(r) on these networks are presented in
Figures \ref{pol_mat}, \ref{cal_mat} and \ref{sim_mat}, respectively. The
results for OCCAM and RBR(K) are not shown due to their similarity with SCORE
and RBR(r), respectively. We can see that
RBR(r) is indeed competitive on identifying good communities.

\begin{figure}[!htb]
\centering
	\includegraphics[width=.31\textwidth]{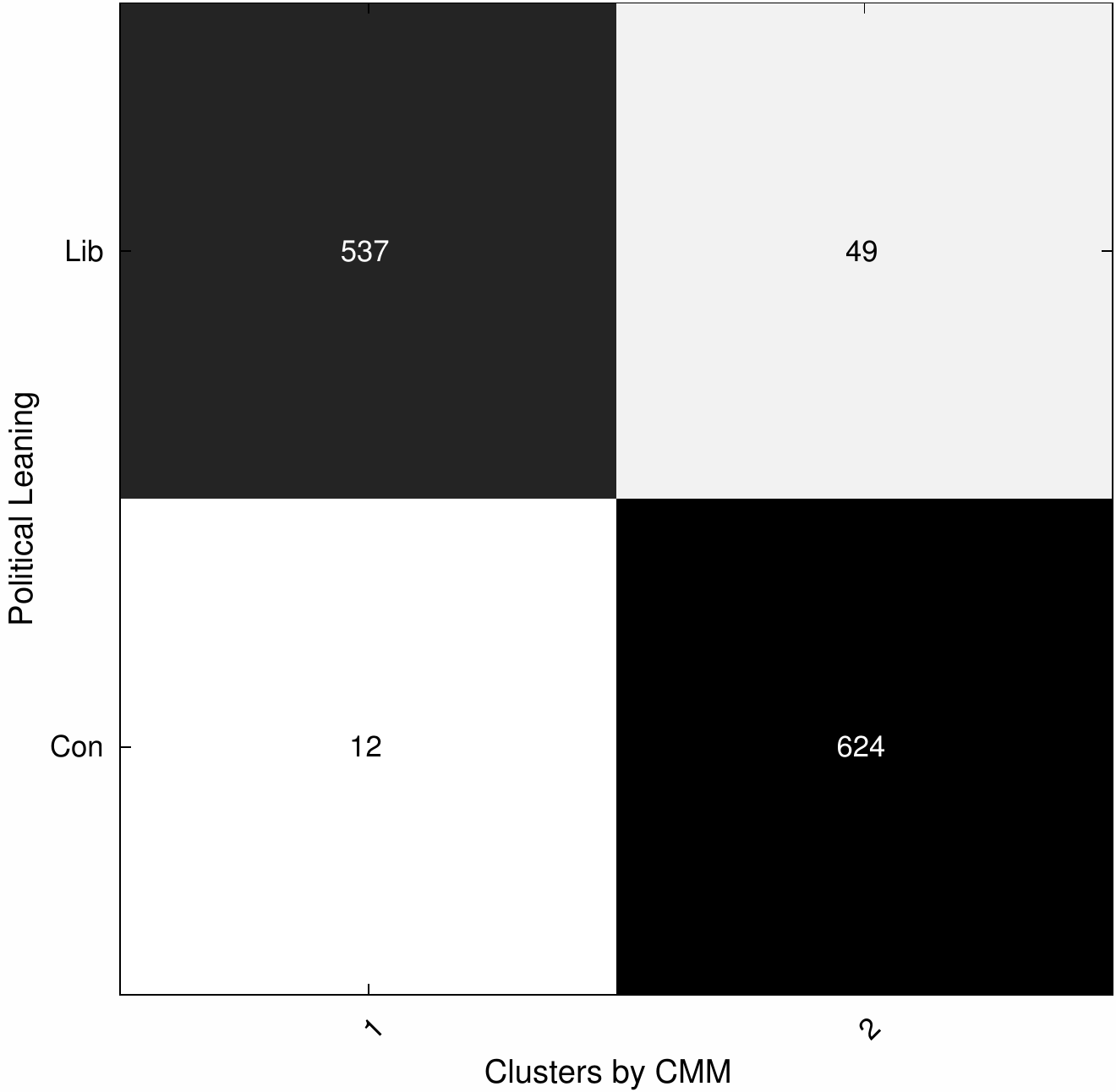}\hspace{1.1ex}
	\includegraphics[width=.31\textwidth]{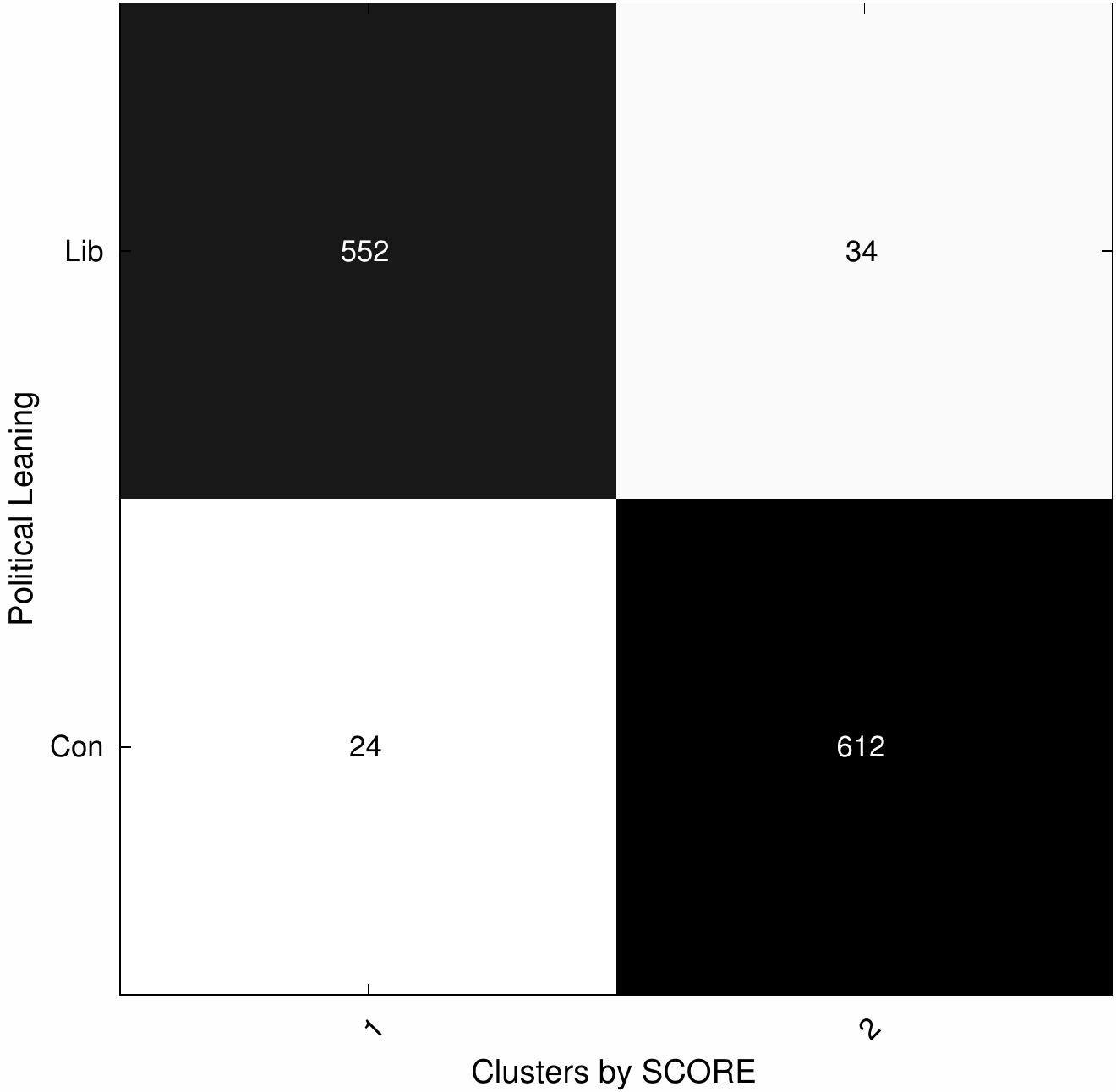}\hspace{1.1ex}
	\includegraphics[width=.31\textwidth]{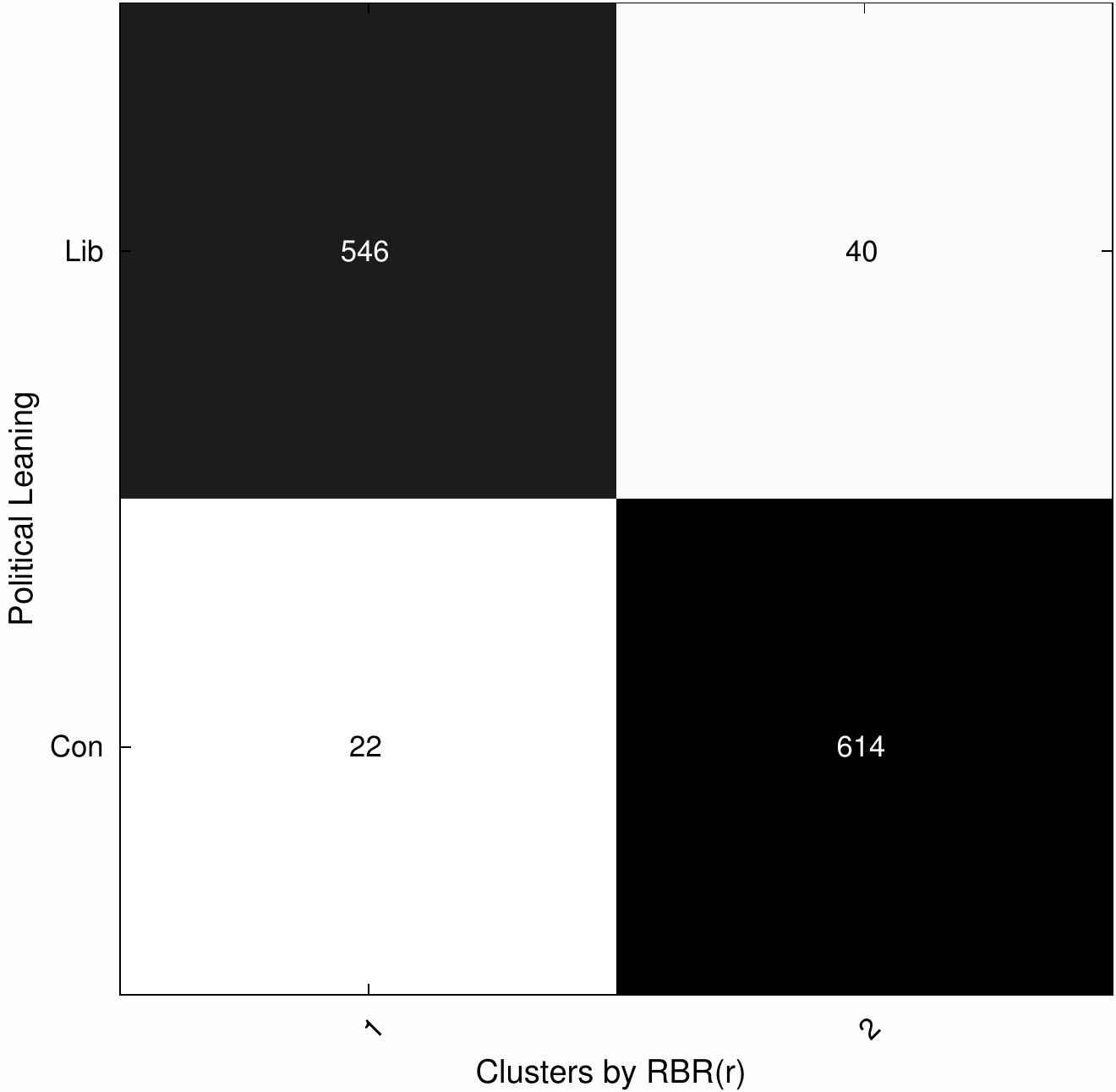}
	\caption{The confusion matrices of CMM, SCORE and RBR(r) on Polblogs network. }
\label{pol_mat}
\end{figure}

\begin{figure}[!htb]
\centering
	\includegraphics[width=.31\textwidth]{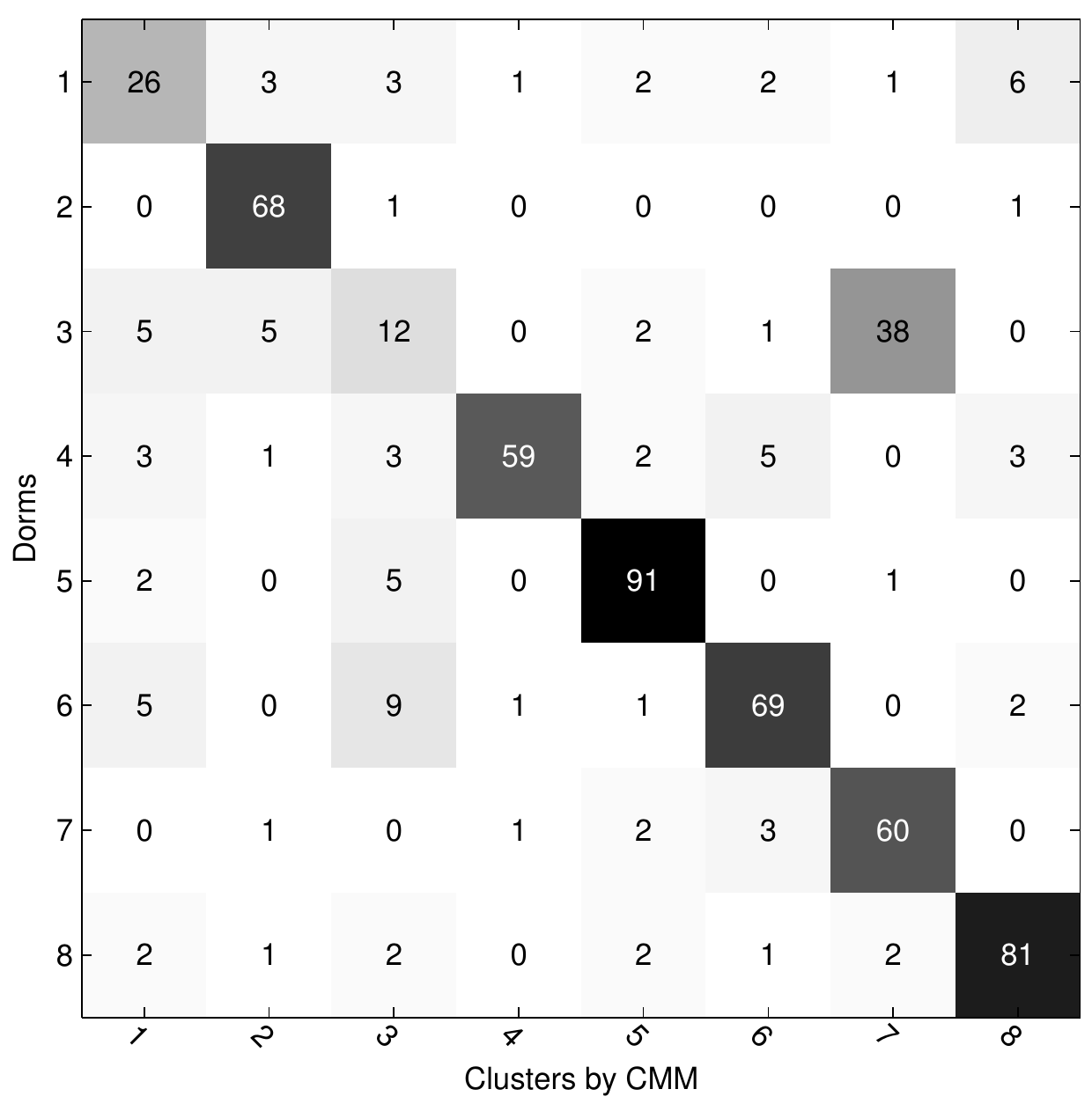}\hspace{1.1ex}
	\includegraphics[width=.31\textwidth]{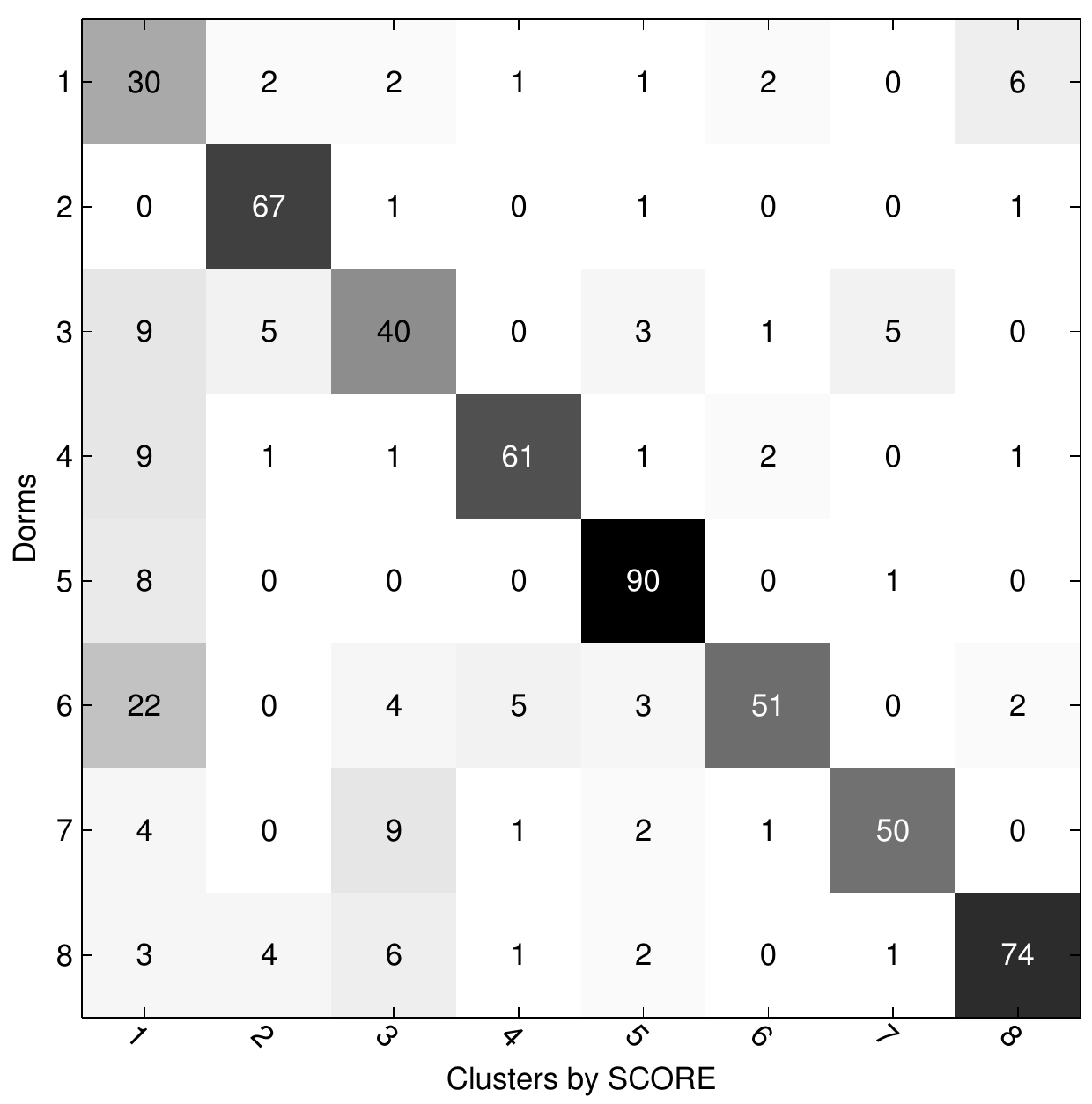}\hspace{1.1ex}
	\includegraphics[width=.31\textwidth]{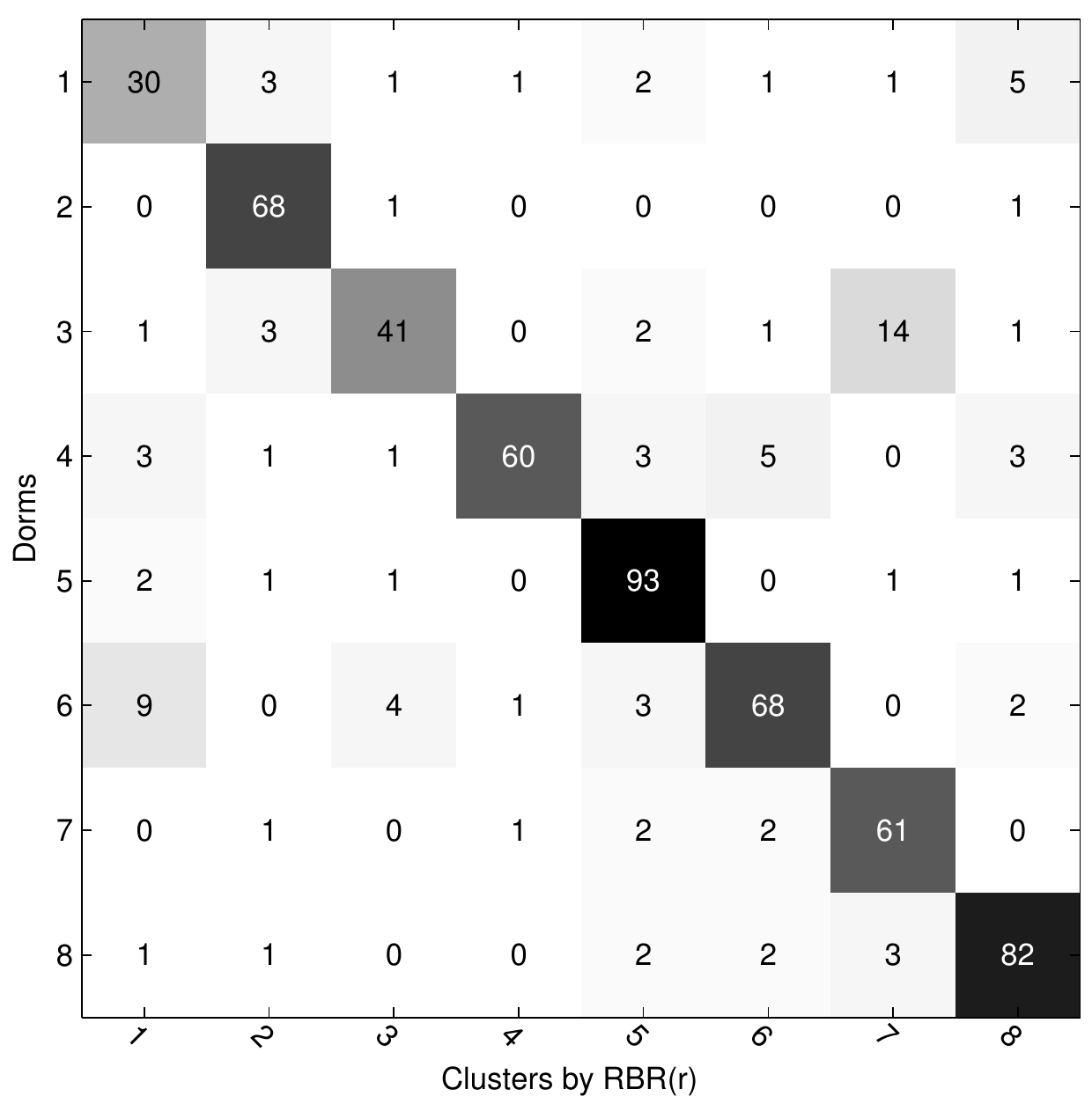}
	\caption{The confusion matrices of CMM, SCORE and RBR(r) on Caltech network. }
\label{cal_mat}
\end{figure}

\begin{figure}[!htb]
	\centering
	\includegraphics[width=.31\textwidth]{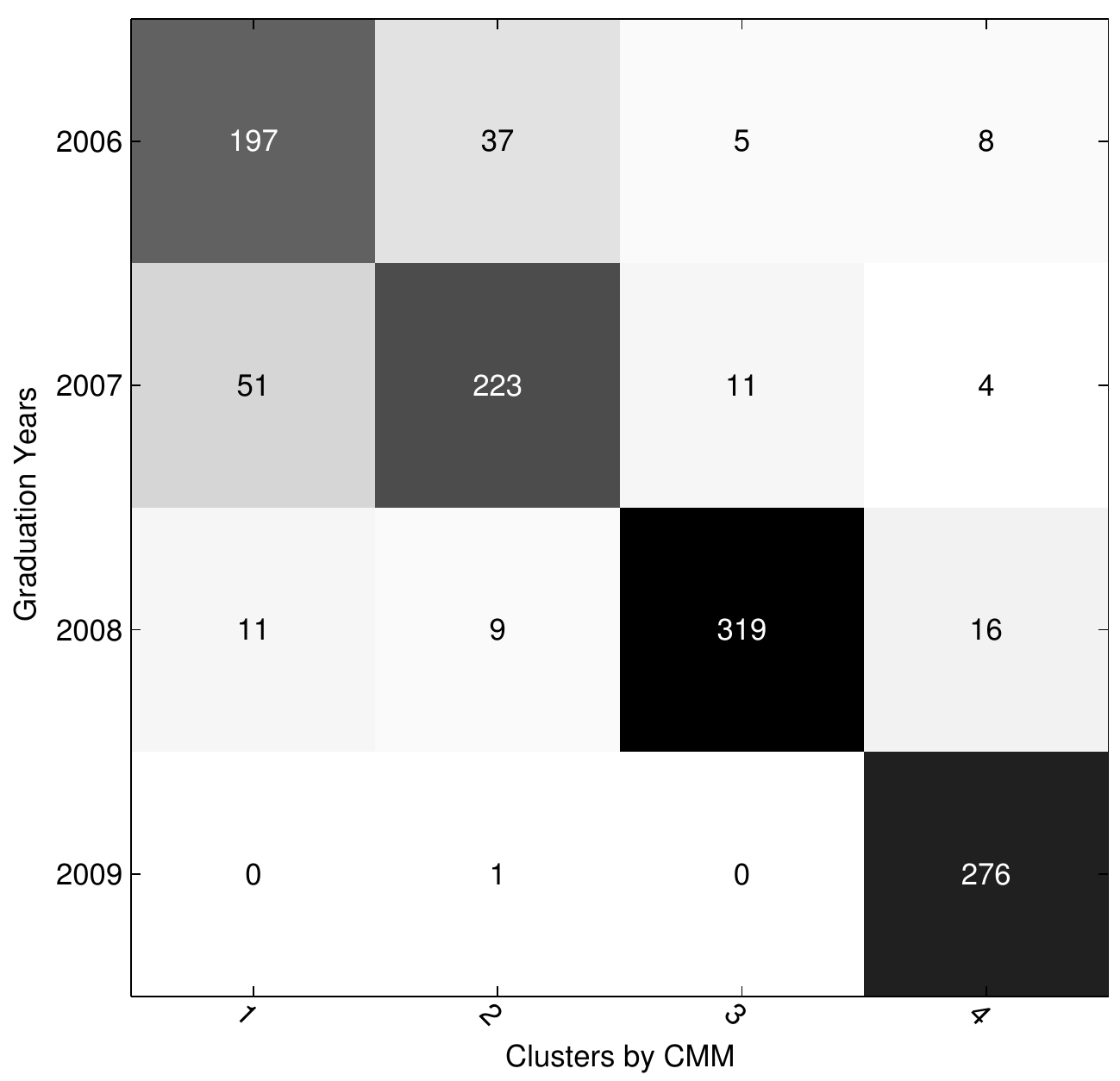}\hspace{1.1ex}
	\includegraphics[width=.31\textwidth]{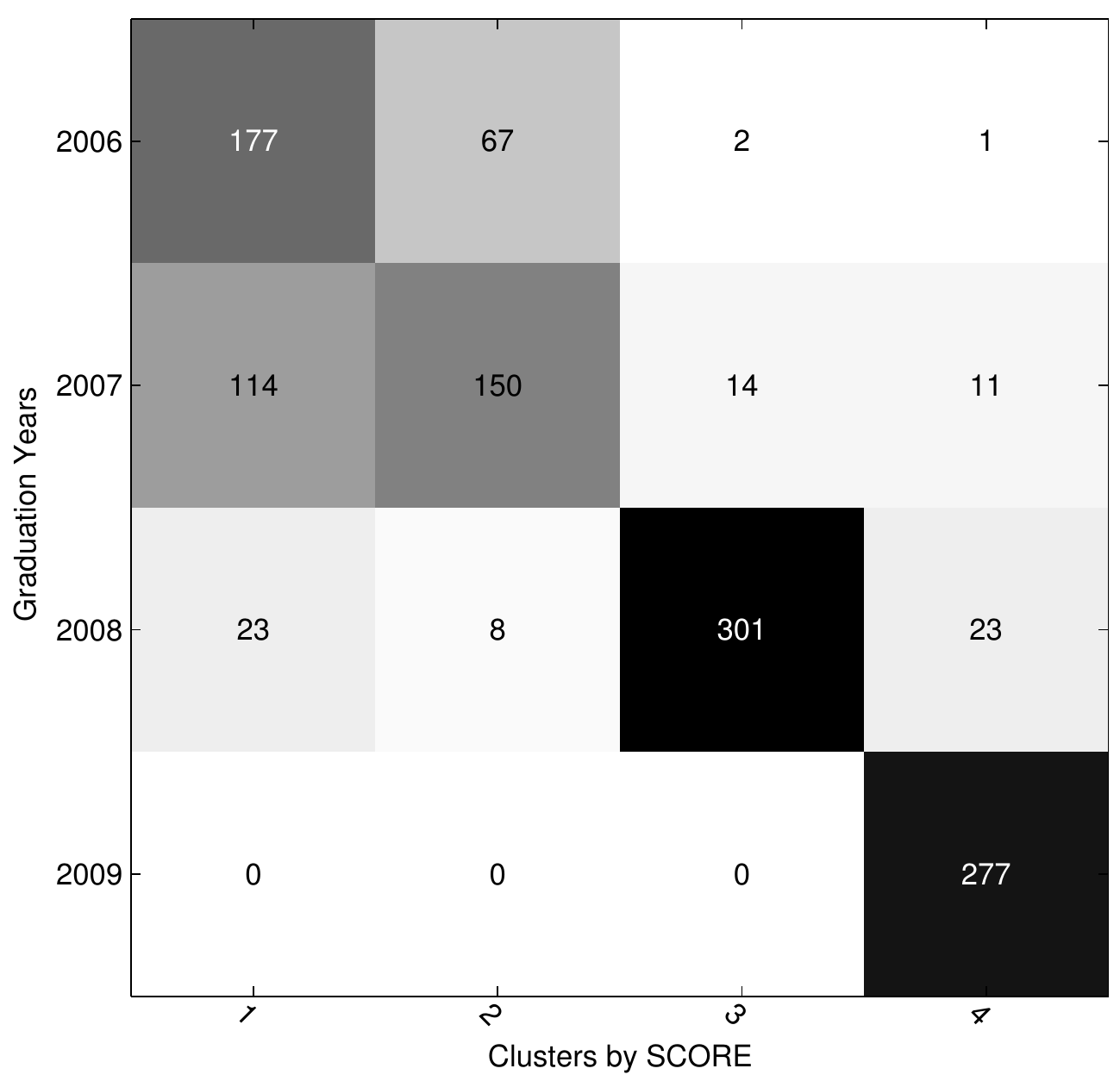}\hspace{1.1ex}
	\includegraphics[width=.31\textwidth]{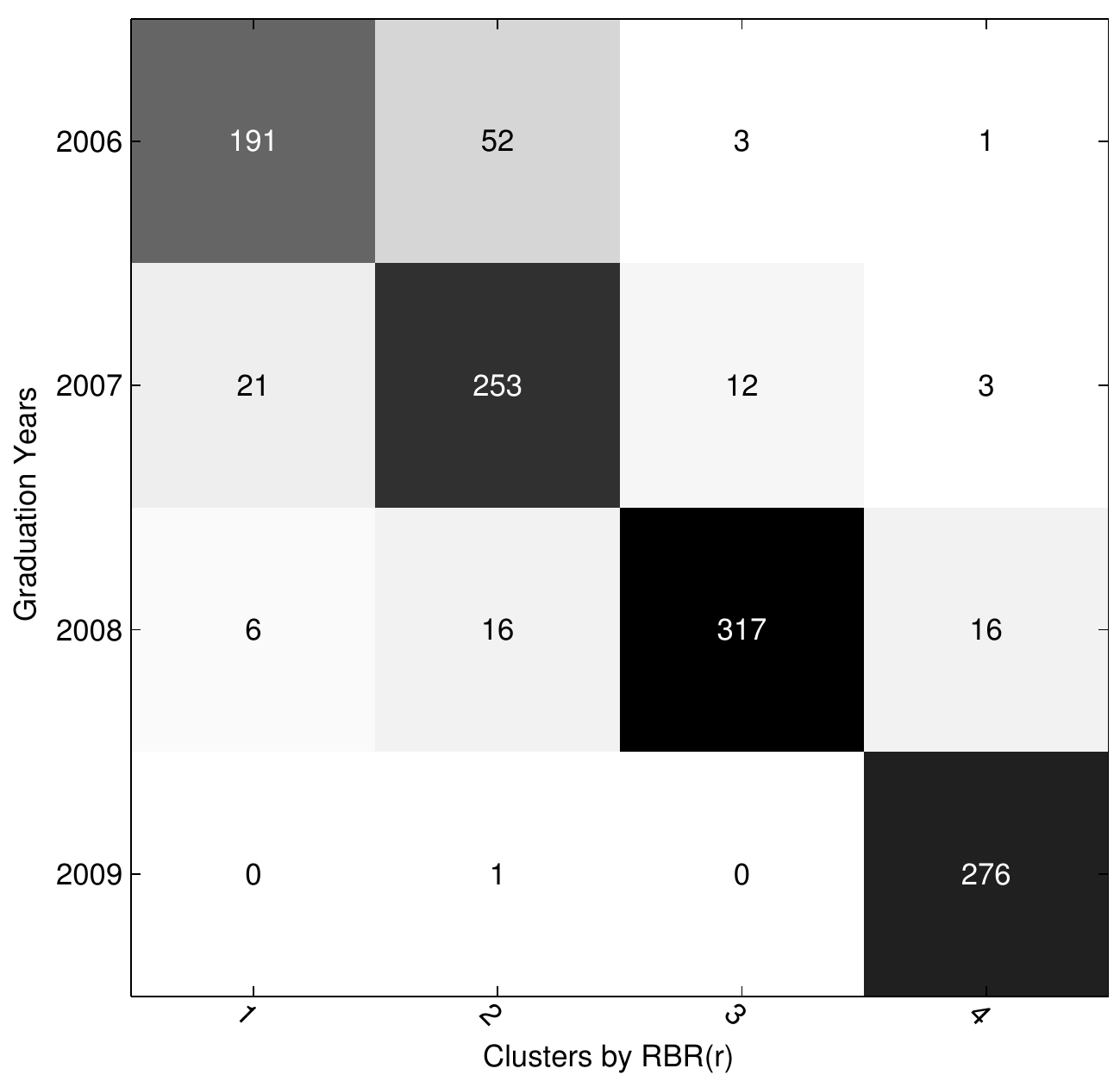}
	\caption{The confusion matrices of CMM, SCORE and RBR(r) on Simmons College network.}
	\label{sim_mat}
\end{figure}


\subsection{Results on Large-Scaled Data}
For large-sized network, we choose ten undirected
networks from Stanford Large Network Dataset Collection \cite{snapnets}
and UF Sparse Matrix Collection \cite{UFSparse}. A detailed description
of these networks can be found in Table \ref{tab:networks}. The number
of nodes in the dataset varies from 36 thousand to 50 million. Due to the size of these networks, we run the RBR(r) algorithm on these networks for only once. 
\begin{table}[h]
  \setlength{\tabcolsep}{2pt} 
\caption{Large-scale networks}
\label{tab:networks}
\centering
{\small
\begin{tabular}{|l||r|r|l|}
\hline
Name        & nodes   & edges     & Descriptions         \\ \hline
amazon      & 334,863 & 925,872   & Amazon product network \\
DBLP        & 317,080 & 1,049,866 & DBLP collaboration network \\
email-Enron & 36,692  & 183,831   & Email communication network from Enron \\
loc-Gowalla & 196,591 & 950,327   & Gowalla location based online social network \\
loc-Brightkite & 58,228 & 214,078 & Brightkite location based online social network \\
youtube     & 1,134,890 & 2,987,624 & Youtube online social network \\
LiveJournal & 3,997,962 & 34,681,189 & LiveJournal online social network \\
Delaunay\_n24 & 16,777,216 & 50,331,601 & Delaunay triangulations of random points in the plane \\     
road\_usa   & 23,947,347 & 28,854,312 & Road network of USA \\
europe\_osm & 50,912,018 & 54,054,660 & Road network of Europe \\ \hline
\end{tabular}
}
\end{table}

For these large networks, we do not have the true partition. Therefore we employ three metrics to evaluate the 
quality of the clustering.

\begin{itemize}
\item \textbf{CC}: \emph{Cluster Coefficient(CC)} is defined as
\begin{equation}
CC=\frac{1}{nC}\sum_{i=1}^{nC}\left(
\frac{1}{|C_i|}\sum_{v\in C_i}\frac{2|\{e_{ts}\in E~:~v_t,v_s\in N(v)\cap C_i\}|}{d(v)(d(v)-1)}
\right)
\end{equation}
where $nC$ is total number of detected communities;
$|C_i|$ is the cardinality of the community $C_i$;
$N(v)$ is the set that contains all neighbors of 
node $v$.
Thus, high CC means that the connections in each cluster
are dense. That is, it assumes that for any $v$ within class $C$,
the neighborhood of $v$ should have the structure
of a complete graph. In some sparse graphs such as road nets, the CC value may be small.

\item \textbf{S}: \emph{Strength(S)} shows that a
community is valid if most nodes have the same label
 as their neighbors. Denote $d(v)^{in}$ and $d(v)^{out}$ as the degrees inside and outside the
community $C$ containing $v$. If for all $v$ in $C$,
we have $d(v)^{in}>d(v)^{out}$, then $C$ is a strong
community; if $\sum_{v\in C}d(v)^{in}>\sum_{v\in C}d(v)^{out}$, then $C$ is a weak community; otherwise
$C$ not valid as a community. Consequently, $S$ is defined as
\begin{equation}
S=\frac{1}{nC}\sum_{i=1}^{nC}\mathrm{score}(C_i),
\end{equation}
where
$$\mathrm{score}(C_i)=\begin{cases}
1, & C_i\mbox{ is strong,}\\
0.5, & C_i \mbox{ is weak,} \\
0, & C_i \mbox{ is invalid.}
\end{cases}$$
A cluster with a high $S$ implies that the neighbouring
nodes tend to have the same label, which is true for
most graphs with community structures.
\item \textbf{Q}: \emph{Modularity(Q)} is a widely used
measure in community detection. In our RBR algorithm,
the modularity is the absolute value of the objective function value modulo a constant factor $2|E|$.
\end{itemize}

We first  evaluate the performance of RBR(r)  with respect to different values of $k$, i.e., the expected
number of communities. Specifically, for each $k\in\{ 5, 10, 20, 30, 50, 100,
200\}$, we perform a fixed number of iterations and report the corresponding
values of the  metrics (CC, S, Q). The paramter $p$ is set to $5$ in all cases. The results are shown in Figure \ref{comp-k}.
The modualarity $Q$ increases as $k$ becomes larger but it becomes almost the
same when $k\geqslant 50$. From an optimization perspective, a larger $k$ means a larger feasible
domain. Therefore, a better value should be expected. On the other hand, the
values of CC and S are almost the same for all $k$. 

\begin{figure}[!htpb]
\caption{Metrics with respect to different values of $k$} \label{comp-k}
\centering
\hfill
\subfigure[amazon]{
	\begin{tikzpicture}[scale=0.45]
	\pgfplotstableread[row sep=\\]{
k    CC    S    Q \\ 
5    0.456    0.500    0.686 \\ 
10    0.454    0.500    0.768 \\ 
20    0.458    0.500    0.827 \\ 
30    0.459    0.500    0.854 \\ 
50    0.463    0.500    0.881 \\ 
100    0.467    0.500    0.895 \\ 
200    0.477    0.500    0.899 \\ 
}\mydata
\begin{axis}[
    ybar=1pt,
    grid,
    width=1\textwidth,
    height=0.64\textwidth,
    bar width=8pt,
    legend pos=north west,
    symbolic x coords={5,10,20,30,50,100,200},
    xtick=data,
    x tick label style={rotate=45, anchor=east},
    xlabel={},
    ylabel={Value},
    title={},
    ymin=0,
    ymax=1,
]
    \addplot[fill=cyan] table[x=k,y=CC]{\mydata};
    \addplot[fill=blue!70!black] table[x=k,y=S]{\mydata};
    \addplot[fill=yellow!80] table[x=k,y=Q]{\mydata};
    \legend{CC,S,Q};
\end{axis}
	\end{tikzpicture}
}
\hfill
\subfigure[DBLP]{
	\begin{tikzpicture}[scale=0.45]
	\pgfplotstableread[row sep=\\]{
k    CC    S    Q \\ 
5    0.673    0.500    0.615 \\ 
10    0.669    0.500    0.700 \\ 
20    0.672    0.500    0.754 \\ 
30    0.678    0.500    0.782 \\ 
50    0.680    0.500    0.799 \\ 
100    0.679    0.500    0.801 \\ 
200    0.684    0.500    0.806 \\ 
}\mydata
\begin{axis}[
    ybar=1pt,
    grid,
    width=1\textwidth,
    height=0.64\textwidth,
    bar width=8pt,
    legend pos=north west,
    symbolic x coords={5,10,20,30,50,100,200},
    xtick=data,
    x tick label style={rotate=45, anchor=east},
    xlabel={},
    ylabel={Value},
    title={},
    ymin=0,
    ymax=1,
]
    \addplot[fill=cyan] table[x=k,y=CC]{\mydata};
    \addplot[fill=blue!70!black] table[x=k,y=S]{\mydata};
    \addplot[fill=yellow!80] table[x=k,y=Q]{\mydata};
    \legend{CC,S,Q};
\end{axis}
	\end{tikzpicture}
}

\hfill
\subfigure[youtube]{
	\begin{tikzpicture}[scale=0.45]
	\pgfplotstableread[row sep=\\]{
k    CC    S    Q \\ 
5    0.603    0.500    0.592 \\ 
10    0.603    0.500    0.695 \\ 
20    0.600    0.500    0.714 \\ 
30    0.601    0.500    0.712 \\ 
50    0.609    0.500    0.712 \\ 
100    0.614    0.490    0.710 \\ 
200    0.607    0.500    0.689 \\ 
}\mydata
\begin{axis}[
    ybar=1pt,
    grid,
    width=1\textwidth,
    height=0.64\textwidth,
    bar width=8pt,
    legend pos=north west,
    symbolic x coords={5,10,20,30,50,100,200},
    xtick=data,
    x tick label style={rotate=45, anchor=east},
    xlabel={},
    ylabel={Value},
    title={},
    ymin=0,
    ymax=1,
]
    \addplot[fill=cyan] table[x=k,y=CC]{\mydata};
    \addplot[fill=blue!70!black] table[x=k,y=S]{\mydata};
    \addplot[fill=yellow!80] table[x=k,y=Q]{\mydata};
    \legend{CC,S,Q};
\end{axis}
	\end{tikzpicture}
}
\hfill
\subfigure[LiveJournal]{
	\begin{tikzpicture}[scale=0.45]
	\pgfplotstableread[row sep=\\]{
k    CC    S    Q \\ 
5    0.455    0.500    0.653 \\ 
10    0.452    0.500    0.724 \\ 
20    0.448    0.500    0.737 \\ 
30    0.457    0.500    0.753 \\ 
50    0.462    0.490    0.757 \\ 
100    0.497    0.490    0.754 \\ 
200    0.527    0.500    0.753 \\ 
}\mydata
\begin{axis}[
    ybar=1pt,
    grid,
    width=1\textwidth,
    height=0.64\textwidth,
    bar width=8pt,
    legend pos=north west,
    symbolic x coords={5,10,20,30,50,100,200},
    xtick=data,
    x tick label style={rotate=45, anchor=east},
    xlabel={},
    ylabel={Value},
    title={},
    ymin=0,
    ymax=1,
]
    \addplot[fill=cyan] table[x=k,y=CC]{\mydata};
    \addplot[fill=blue!70!black] table[x=k,y=S]{\mydata};
    \addplot[fill=yellow!80] table[x=k,y=Q]{\mydata};
    \legend{CC,S,Q};
\end{axis}
	\end{tikzpicture}
}

\end{figure}

\begin{figure}[!htpb]
\caption{Metric with respect to different values of $p$} \label{comp-p}
\centering
\hfill
\subfigure[amazon]{
	\begin{tikzpicture}[scale=0.45]
	\pgfplotstableread[row sep=\\]{
k    CC    S    Q \\ 
1    0.426    0.500    0.755 \\ 
2    0.454    0.500    0.848 \\ 
5    0.467    0.500    0.895 \\ 
10    0.468    0.500    0.900 \\ 
20    0.470    0.500    0.901 \\ 
}\mydata
\begin{axis}[
    ybar=2pt,
    grid,
    width=1\textwidth,
    height=0.64\textwidth,
    bar width=12pt,
    legend pos=north west,
    symbolic x coords={1,2,5,10,20},
    xtick=data,
    x tick label style={rotate=45, anchor=east},
    xlabel={},
    ylabel={Value},
    title={},
    ymin=0,
    ymax=1,
]
    \addplot[fill=cyan] table[x=k,y=CC]{\mydata};
    \addplot[fill=blue!70!black] table[x=k,y=S]{\mydata};
    \addplot[fill=yellow!80] table[x=k,y=Q]{\mydata};
    \legend{CC,S,Q};
\end{axis}
	\end{tikzpicture}
}
\hfill
\subfigure[DBLP]{
	\begin{tikzpicture}[scale=0.45]
	\pgfplotstableread[row sep=\\]{
k    CC    S    Q \\ 
1    0.585    0.500    0.665 \\ 
2    0.655    0.500    0.760 \\ 
5    0.679    0.500    0.801 \\ 
10    0.686    0.500    0.809 \\ 
20    0.686    0.500    0.808 \\ 
}\mydata
\begin{axis}[
    ybar=2pt,
    grid,
    width=1\textwidth,
    height=0.64\textwidth,
    bar width=12pt,
    legend pos=north west,
    symbolic x coords={1,2,5,10,20},
    xtick=data,
    x tick label style={rotate=45, anchor=east},
    xlabel={},
    ylabel={Value},
    title={},
    ymin=0,
    ymax=1,
]
    \addplot[fill=cyan] table[x=k,y=CC]{\mydata};
    \addplot[fill=blue!70!black] table[x=k,y=S]{\mydata};
    \addplot[fill=yellow!80] table[x=k,y=Q]{\mydata};
    \legend{CC,S,Q};
\end{axis}
	\end{tikzpicture}
}

\hfill
\subfigure[road\_usa]{
	\begin{tikzpicture}[scale=0.45]
	\pgfplotstableread[row sep=\\]{
k    CC    S    Q \\ 
1    0.216    0.500    0.736 \\ 
2    0.216    0.500    0.766 \\ 
5    0.216    0.500    0.908 \\ 
10    0.216    0.500    0.915 \\ 
20    0.216    0.500    0.916 \\ 
}\mydata
\begin{axis}[
    ybar=2pt,
    grid,
    width=1\textwidth,
    height=0.64\textwidth,
    bar width=12pt,
    legend pos=north west,
    symbolic x coords={1,2,5,10,20},
    xtick=data,
    x tick label style={rotate=45, anchor=east},
    xlabel={},
    ylabel={Value},
    title={},
    ymin=0,
    ymax=1,
]
    \addplot[fill=cyan] table[x=k,y=CC]{\mydata};
    \addplot[fill=blue!70!black] table[x=k,y=S]{\mydata};
    \addplot[fill=yellow!80] table[x=k,y=Q]{\mydata};
    \legend{CC,S,Q};
\end{axis}
	\end{tikzpicture}
}
\hfill
\subfigure[Delaunay\_n24]{
	\begin{tikzpicture}[scale=0.45]
	\pgfplotstableread[row sep=\\]{
k    CC    S    Q \\ 
1    0.298    0.500    0.730 \\ 
2    0.296    0.500    0.733 \\ 
5    0.365    0.500    0.868 \\ 
10    0.369    0.500    0.876 \\ 
20    0.370    0.500    0.877 \\ 
}\mydata
\begin{axis}[
    ybar=2pt,
    grid,
    width=1\textwidth,
    height=0.64\textwidth,
    bar width=12pt,
    legend pos=north west,
    symbolic x coords={1,2,5,10,20},
    xtick=data,
    x tick label style={rotate=45, anchor=east},
    xlabel={},
    ylabel={Value},
    title={},
    ymin=0,
    ymax=1,
]
    \addplot[fill=cyan] table[x=k,y=CC]{\mydata};
    \addplot[fill=blue!70!black] table[x=k,y=S]{\mydata};
    \addplot[fill=yellow!80] table[x=k,y=Q]{\mydata};
    \legend{CC,S,Q};
\end{axis}
	\end{tikzpicture}
}

\end{figure}
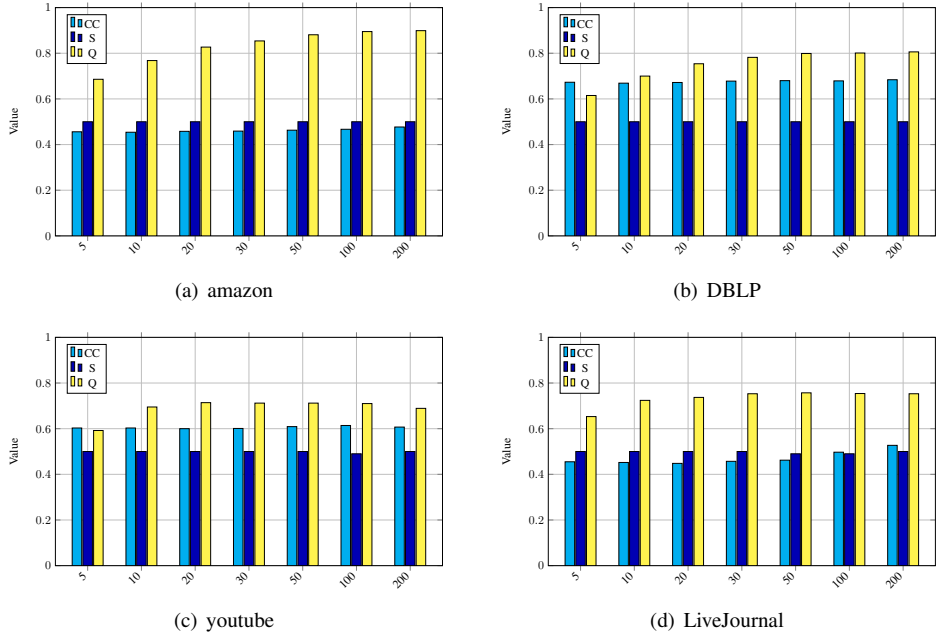

We next compare the performance of RBR(r) with respect to different values of
$p$ in the $\ell_0$ constraint
over each subproblem. By taking advantage of the sparsity,  RBR  is able to solve  problems
 of very large scale since the storage of $U$ is $\mathcal{O}(np)$. For each
$p\in \{1, 2, 5, 10, 20\}$, we set $k=100$ and perform a fixed number of
iterations. The results are shown in Figure \ref{comp-p}. Ideally,  RBR
 tends to search solutions in a larger space for
the subproblem if a larger
$p$ is applied.   From the figure, we
can observe  that  RBR  performs
better when a larger $p$ is used. The metrics become similar when $p\geqslant5$
but they are better than the results with values $p=1$ and $p=2$.  

Finally we compare our parallel RBR with  LOUVAIN,
which runs extremely fast and usually provides good results. It is worth
mentioning that LOUVAIN  can begin with different initial solutions.
By default it starts from a trivial classification that each node belongs to
a community that only contains itself. It merges  clusters in order
to obtain a higher modularity.  LOUVAIN  terminates if it is unable
to increase the modularity or the improvement is not significant. In this sense, there is no guarantee on the number of  communities it finally provides.
On the other hand, one can limit the maximal number of communities by
letting  LOUVAIN  start from a partition with no more than $k$
communities. Since  LOUVAIN  always reduces the number of
communities by merging, it is possible to obtain a low rank solution.
However, the quality of the solution may be highly dependent on the initial
values. Without utilizing much prior knowledge of these networks, we
set the initial value of both algorithms to be a random partition with
at most $k$ labels.

A summary of the computational results are shown in Table \ref{tab:large-all}. For each
algorithm, three sets of parameters are tested. For example, RBR(100, 1)
denotes that the desired number of communities $k$ is set to 100 and $p$ is
set to 1; LV(500) means that  LOUVAIN  begins with a randomly assigned
partition with at most 500 communities. LV(0), however, means that we
launch  LOUVAIN  with a trivial partition, in which case it may not produce low
rank solutions. In the table, $k_0$ stands for the actual
number of communities that the algorithm finally produces. If $k$ is given, then
$k_0$ will not exceed $k$. The CPU time of each algorithm is also reported.

\begin{table}[!htpb]
\centering
\scriptsize
\caption{Comparison between RBR and LOUVAIN algorithm under different settings} \label{tab:large-all}
\begin{tabular}{|r||rrrrr||rrrrr||}\hline 
\multirow{2}{*}{method} & \multicolumn{5}{c||}{amazon} & \multicolumn{5}{c||}{youtube}\\ 
\cline{2-11}
& $k_0$ & CC & S & Q & Time & $k_0$ & CC & S & Q & Time\\ \hline
RBR(100,1) & 100 & 0.426 & 0.500 & 0.755 & 3.5 & 100 & 0.589 & 0.480 & 0.623 & 12.5\\ 
RBR(20,5) & 20 & 0.458 & 0.500 & 0.827 & 3.0 & 20 & 0.600 & 0.500 & 0.714 & 10.2\\ 
RBR(100,5) & 100 & 0.467 & 0.500 & 0.895 & 5.0 & 98 & 0.614 & 0.490 & 0.710 & 22.6\\ 
LV(0) & 230 & 0.531 & 0.598 & 0.926 & 1.7 & 28514 & 0.930 & 0.159 & 0.723 & 8.7\\ 
LV(500) & 165 & 0.427 & 0.500 & 0.757 & 1.7 & 95 & 0.593 & 0.495 & 0.646 & 7.0\\ 
LV(2000) & 183 & 0.435 & 0.500 & 0.775 & 1.7 & 93 & 0.587 & 0.495 & 0.658 & 6.0\\ 
\hline 
\hline 
\multirow{2}{*}{method} & \multicolumn{5}{c||}{DBLP} & \multicolumn{5}{c||}{LiveJournal}\\ 
\cline{2-11}
& $k_0$ & CC & S & Q & Time & $k_0$ & CC & S & Q & Time\\ \hline
RBR(100,1) & 100 & 0.585 & 0.500 & 0.665 & 3.1 & 100 & 0.425 & 0.500 & 0.668 & 39.2\\ 
RBR(20,5) & 20 & 0.672 & 0.500 & 0.754 & 2.9 & 20 & 0.448 & 0.500 & 0.737 & 43.4\\ 
RBR(100,5) & 100 & 0.679 & 0.500 & 0.801 & 5.6 & 98 & 0.497 & 0.490 & 0.754 & 64.5\\ 
LV(0) & 187 & 0.742 & 0.618 & 0.820 & 2.9 & 2088 & 0.729 & 0.877 & 0.749 & 181.5\\ 
LV(500) & 146 & 0.583 & 0.500 & 0.673 & 2.1 & 101 & 0.433 & 0.500 & 0.665 & 81.9\\ 
LV(2000) & 155 & 0.578 & 0.500 & 0.681 & 2.3 & 113 & 0.447 & 0.500 & 0.687 & 172.7\\ 
\hline 
\hline 
\multirow{2}{*}{method} & \multicolumn{5}{c||}{email-Enron} & \multicolumn{5}{c||}{loc-Brightkite}\\ 
\cline{2-11}
& $k_0$ & CC & S & Q & Time & $k_0$ & CC & S & Q & Time\\ \hline
RBR(100,1) & 100 & 0.589 & 0.400 & 0.559 & 0.3 & 100 & 0.485 & 0.360 & 0.596 & 0.5\\ 
RBR(20,5) & 20 & 0.708 & 0.500 & 0.605 & 0.3 & 20 & 0.516 & 0.500 & 0.667 & 0.5\\ 
RBR(100,5) & 100 & 0.782 & 0.585 & 0.605 & 0.6 & 100 & 0.552 & 0.510 & 0.675 & 0.8\\ 
LV(0) & 1245 & 0.935 & 0.974 & 0.597 & 0.2 & 732 & 0.829 & 0.939 & 0.687 & 0.3\\ 
LV(500) & 45 & 0.705 & 0.600 & 0.581 & 0.2 & 47 & 0.513 & 0.500 & 0.637 & 0.3\\ 
LV(2000) & 331 & 0.894 & 0.899 & 0.604 & 0.2 & 79 & 0.653 & 0.671 & 0.661 & 0.4\\ 
\hline 
\hline 
\multirow{2}{*}{method} & \multicolumn{5}{c||}{loc-Gowalla} & \multicolumn{5}{c||}{Delaunay\_n24}\\ 
\cline{2-11}
& $k_0$ & CC & S & Q & Time & $k_0$ & CC & S & Q & Time\\ \hline
RBR(100,1) & 100 & 0.439 & 0.470 & 0.637 & 1.7 & 100 & 0.298 & 0.500 & 0.730 & 190.7\\ 
RBR(20,5) & 20 & 0.459 & 0.500 & 0.694 & 1.7 & 20 & 0.372 & 0.500 & 0.837 & 172.9\\ 
RBR(100,5) & 98 & 0.497 & 0.475 & 0.696 & 4.3 & 100 & 0.365 & 0.500 & 0.868 & 252.5\\ 
LV(0) & 850 & 0.592 & 0.761 & 0.705 & 1.2 & 354 & 0.430 & 0.500 & 0.990 & 189.1\\ 
LV(500) & 60 & 0.442 & 0.500 & 0.658 & 1.3 & 500 & 0.287 & 0.500 & 0.717 & 792.7\\ 
LV(2000) & 65 & 0.460 & 0.500 & 0.672 & 1.1 & 981 & 0.286 & 0.500 & 0.716 & 1390.5\\ 
\hline 
\hline 
\multirow{2}{*}{method} & \multicolumn{5}{c||}{europe\_osm} & \multicolumn{5}{c||}{road\_usa}\\ 
\cline{2-11}
& $k_0$ & CC & S & Q & Time & $k_0$ & CC & S & Q & Time\\ \hline
RBR(100,1) & 100 & 0.042 & 0.500 & 0.790 & 689.9 & 100 & 0.216 & 0.500 & 0.736 & 274.3\\ 
RBR(20,5) & 20 & 0.042 & 0.500 & 0.822 & 574.5 & 20 & 0.216 & 0.500 & 0.841 & 239.6\\ 
RBR(100,5) & 100 & 0.042 & 0.500 & 0.894 & 717.4 & 100 & 0.216 & 0.500 & 0.908 & 297.4\\ 
LV(0) & 3057 & 0.038 & 0.501 & 0.999 & 481.8 & 1569 & 0.225 & 0.510 & 0.998 & 303.2\\ 
LV(500) & 500 & 0.041 & 0.500 & 0.663 & 243.7 & 500 & 0.216 & 0.500 & 0.663 & 153.6\\ 
LV(2000) & 1999 & 0.041 & 0.500 & 0.663 & 337.3 & 1367 & 0.216 & 0.500 & 0.662 & 206.1\\ 
\hline 
\end{tabular}
\end{table}

From the perspective of  solution qualities, LV(0) outperforms  other algorithms
and settings. Its CC, S, Q are  high on most networks. Especially for the
last three large networks, the modularity Q is over 0.99. However, it is hard to
control $k_0$ in LV(0). It is
acceptable on networks such as DBLP and Delaunay\_n24.  It may produce over one thousand
communities and $k_0$ is over $28000$ in the youtube network. 
Although LV(500) and LV(2000) can return partitions with a small $k_0$, their
modularity is not as good as LV(0).

We can see that  RBR(100,5) almost always returns a
good moduality. For instance, it is better than LV(0) on email-Enron.  
 Note that RBR(20,5) is also very
competitive due to its low rank solutions without sacrificing the modularity too much. These facts indicate that  RBR  is better than  LOUVAIN if a small number of communities are needed.
 Additionally, it can be observed that RBR(100,1) can produce similar results as
LV(500) and LV(2000) do. In fact, the procedure of these algorithms
are similar in the sense of finding a new label for a given node and increasing the modularity Q.
The difference is that  LOUVAIN  only considers the labels of the node's neighbours while
 RBR  searches for the new label among all $k$ possibilities. If RBR is modified
 properly,  it  can also utilize the network
structure. 

 LOUVAIN  is faster in terms of the running time.  Each step of  LOUVAIN  only considers at most $d_{\mathrm{max}}$
nodes, where $d_{\mathrm{max}}$ is the maximum of the vertex degrees, while
 RBR  has to sort all $k$ possible labels and choose at most $p$ best
candidates. However, RBR(20,5) is still competitive on most graphs in terms of both speed and the capability of identifying good communities. It is the fastest one on networks such as LiveJournal and Delaunay\_n24.  Of course, it is possible for  RBR  to consider fewer labels
  according to the structure of the network in order to reduce the computational
  cost of sorting. 
  
  We should point out that  LOUVAIN  is run in a single
  threaded mode since its multi-threaded version is not
  available.  On the other hand,  RBR(r) is an asynchronous parallel method. Although its
  theoretical property is not clear, it works
  fine in all of our experiments. Thus, RBR(r) is  useful on multi-core machines.   
 The speedup of RBR(r) is shown in Figure \ref{fig:parallel}  when utilizing multiple
computing threads. We can see that our RBR algorithm has almost a linear
acceleration rate thanks to asynchronous parallelization. The average speedup is
close to $8$ using $12$ threads. In particular, it reaches $10$ times speedup 
on the loc-Brightkite network. 

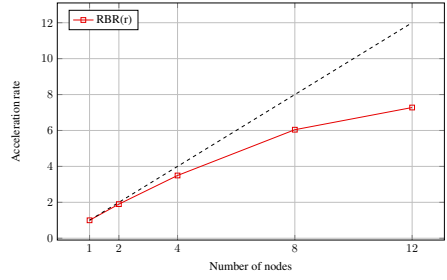
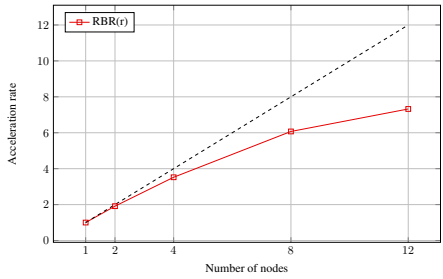
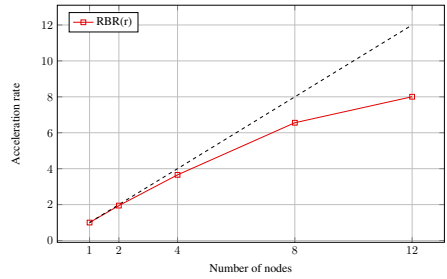
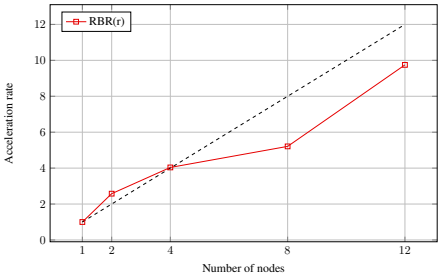
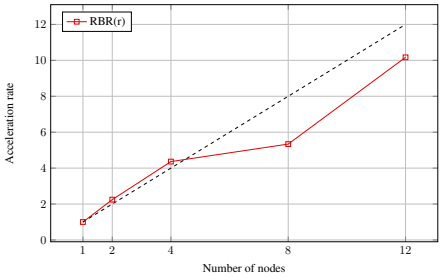
\begin{figure}[!htbp]
\centering
 \hfill
    \subfigure[amazon]{
    \begin{tikzpicture}[scale=0.45]
    \pgfplotstableread[row sep=\\]{
thread	DCRBR \\
1	1.000000 \\
2	1.902587 \\ 
4	3.538358 \\
8	6.074125 \\
12	7.257014 \\
}\mydata

\begin{axis}[
    grid,
    width=\textwidth,
    height=0.66\textwidth,
    legend pos=north west,
    xtick=data,
    ylabel={Acceleration rate},
    xlabel={Number of nodes}
    ]
    \addplot[thick,mark=square,draw=red!90!black] table[x=thread,y=DCRBR]{\mydata};

    \addplot[color=black,dashed] table[x=thread,y=thread]{\mydata};
    \legend{RBR(r)}
\end{axis}

    \end{tikzpicture}
  }
  \hfill
    \subfigure[youtube]{
    \begin{tikzpicture}[scale=0.45]
    \pgfplotstableread[row sep=\\]{
thread	DCRBR \\
1	1.000000 \\
2	1.899377 \\
4	3.494288 \\
8	6.042552 \\
12	7.278053 \\
}\mydata

\begin{axis}[
    grid,
    width=\textwidth,
    height=0.66\textwidth,
    legend pos=north west,
    xtick=data,
    ylabel={Acceleration rate},
    xlabel={Number of nodes}
    ]
    \addplot[thick,mark=square,draw=red!90!black] table[x=thread,y=DCRBR]{\mydata};

    \addplot[color=black,dashed] table[x=thread,y=thread]{\mydata};
    \legend{RBR(r)}
\end{axis}

    \end{tikzpicture}
  }

\vspace{.75cm}

\centering
 \hfill
    \subfigure[DBLP]{
    \begin{tikzpicture}[scale=0.45]
    \pgfplotstableread[row sep=\\]{
thread	DCRBR \\
1	1.000000 \\
2	1.912410 \\
4	3.521510 \\
8	6.071580 \\
12	7.323253 \\
}\mydata

\begin{axis}[
    grid,
    width=\textwidth,
    height=0.66\textwidth,
    legend pos=north west,
    xtick=data,
    ylabel={Acceleration rate},
    xlabel={Number of nodes}
    ]
    \addplot[thick,mark=square,draw=red!90!black] table[x=thread,y=DCRBR]{\mydata};

    \addplot[color=black,dashed] table[x=thread,y=thread]{\mydata};
    \legend{RBR(r)}
\end{axis}

    \end{tikzpicture}
  }
  \hfill
    \subfigure[LiveJournal]{
    \begin{tikzpicture}[scale=0.45]
    \pgfplotstableread[row sep=\\]{
thread	DCRBR \\
1	1.000000 \\
2	1.944175 \\
4	3.654660 \\
8	6.557224 \\
12	8.004910 \\
}\mydata

\begin{axis}[
    grid,
    width=\textwidth,
    height=0.66\textwidth,
    legend pos=north west,
    xtick=data,
    ylabel={Acceleration rate},
    xlabel={Number of nodes}
    ]
    \addplot[thick,mark=square,draw=red!90!black] table[x=thread,y=DCRBR]{\mydata};

    \addplot[color=black,dashed] table[x=thread,y=thread]{\mydata};
    \legend{RBR(r)}
\end{axis}

    \end{tikzpicture}
  }
  
\vspace{.75cm}
\centering
 \hfill
    \subfigure[email-Enron]{
    \begin{tikzpicture}[scale=0.45]
    \pgfplotstableread[row sep=\\]{
thread	DCRBR \\
1	1.000000 \\
2	2.577370 \\
4	4.037936 \\
8	5.210017 \\
12	9.747282 \\
}\mydata

\begin{axis}[
    grid,
    width=\textwidth,
    height=0.66\textwidth,
    legend pos=north west,
    xtick=data,
    ylabel={Acceleration rate},
    xlabel={Number of nodes}
    ]
    \addplot[thick,mark=square,draw=red!90!black] table[x=thread,y=DCRBR]{\mydata};

    \addplot[color=black,dashed] table[x=thread,y=thread]{\mydata};
    \legend{RBR(r)}
\end{axis}

    \end{tikzpicture}
  }
  \hfill
    \subfigure[loc-Brightkite]{
    \begin{tikzpicture}[scale=0.45]
    \pgfplotstableread[row sep=\\]{
thread	DCRBR \\
1	1.000000 \\
2	2.245357 \\
4	4.364165 \\
8	5.338647 \\
12	10.167972 \\
}\mydata

\begin{axis}[
    grid,
    width=\textwidth,
    height=0.66\textwidth,
    legend pos=north west,
    xtick=data,
    ylabel={Acceleration rate},
    xlabel={Number of nodes}
    ]
    \addplot[thick,mark=square,draw=red!90!black] table[x=thread,y=DCRBR]{\mydata};

    \addplot[color=black,dashed] table[x=thread,y=thread]{\mydata};
    \legend{RBR(r)}
\end{axis}

    \end{tikzpicture}
  }
\caption{Acceleration rate of RBR(r) on large real data} \label{fig:parallel}
\end{figure}

\section{Conclusion}  This paper is concerned with  the community detection problem. We develop a sparse and low-rank relaxation of the
modularity maximization model followed by either the K-means or weighted K-means
clusterings or a direct rounding procedure. A fast  algorithm called RBR with asynchronous parallellization is designed to efficiently
solve the proposed problem. We provide some  non-asymptotically bounds on  the misclassification rate with
high probability under some standard random network assumptions. Numerical experiments are fully consistent with the theoretical bounds. The new proposed method provides a competitive alternative state-of-the-art methods in terms of both misclassification rates and numerical efficiency, in addition to the assured theoretical bounds. 

\textbf{Acknowledgment:} The authors thank Yizhou Wang for testing an early version of the algorithm on Small-Scaled Real World Data.

\bibliographystyle{siam}
\bibliography{Community5.0}

\end{document}